\documentclass{amsart}
\usepackage{amssymb,amsmath, amsthm,latexsym}
\usepackage{graphics}
\usepackage{amscd}
\usepackage{graphics}
\newcommand{\cal}[1]{\mathcal{#1}}
\theoremstyle{plain}
\newtheorem{theo}{Theorem}
\newtheorem{lemma}{Lemma}[section]
\newtheorem{theorem}[lemma]{Theorem}
\newtheorem{proposition}[lemma]{Proposition}
\newtheorem{corollary}[lemma]{Corollary}
\theoremstyle{definition}
\newtheorem{definition}[lemma]{Definition}
\parskip=\bigskipamount

\let\egthree=\phi
\let\phi=\varphi
\let\varphi=\egthree

\begin{document}
\title{Dynamics of
the Teichm\"uller flow on compact invariant sets}
\author{Ursula Hamenst\"adt}
\thanks
{
Research
partially supported by DFG SFB 611}
\date{May 11, 2010}

\begin{abstract}
Let $S$ be
an oriented surface $S$ of genus
$g\geq 0$ with $m\geq 0$ punctures and $3g-3+m\geq 2$.
For a 
compact subset $K$ of the moduli space of area one holomorphic
quadratic differentials for $S$ let 
$\delta(K)$ be the asymptotic growth rate 
of the number of periodic orbits
for the Teichm\"uller flow $\Phi^t$ which are contained
in $K$. We relate
$\delta(K)$ to the topological entropy of the 
restriction of $\Phi^t$ to $K$.
Moreover, we show that 
$\sup_K\delta(K)=6g-6+2m$. 
\end{abstract}

\maketitle

\section{Introduction}

An oriented surface $S$ is called of finite type if $S$ is a
closed surface of genus $g\geq 0$ from which $m\geq 0$
points, so-called \emph{punctures},
have been deleted. We assume that 
$S$ is \emph{nonexceptional}, i.e. that $3g-3+m\geq 2$. This means
that $S$ is not a sphere with at most four
punctures or a torus with at most one puncture.
Since the Euler characteristic of $S$ is negative,
the \emph{Teichm\"uller space} ${\cal T}(S)$
of $S$ is the quotient of the space of all complete hyperbolic
metrics on $S$ of finite volume under the action of the
group of diffeomorphisms of $S$ which are isotopic
to the identity. 

The fibre bundle ${\cal Q}^1(S)$
over ${\cal T}(S)$ of all \emph{holomorphic
quadratic differentials} of area
one can naturally be viewed as the unit cotangent
bundle of ${\cal T}(S)$ for the \emph{Teichm\"uller metric}.
The \emph{Teichm\"uller geodesic flow} $\Phi^t$ on
${\cal Q}^1(S)$ commutes
with the action of the \emph{mapping class group}
${\rm Mod}(S)$ of all isotopy classes of
orientation preserving self-homeomorphisms of $S$.
Thus this flow descends to a flow
on the quotient 
${\cal Q}^1(S)/{\rm Mod}(S)$, again denoted
by $\Phi^t$. This quotient is a non-compact
orbifold.

In his seminal paper \cite{V86}, Veech showed that the asymptotic
growth rate of the number of periodic orbits of the Teichm\"uller
flow $\Phi^t$ on ${\cal Q}^1(S)/{\rm Mod}(S)$ is at least $h=6g-6+2m$ (we use
here a normalization for the Teichm\"uller flow which differs from
the one used by Veech). Recently Eskin and Mirzakhani \cite{EM08}
obtained a sharp counting result: They show that as $r\to
\infty$, the number of periodic orbits for $\Phi^t$ of period at
most $r$ is asymptotic to $e^{hr}/hr$. An earlier partial
result for the Teichm\"uller flow on the space
of abelian differentials is due to Bufetov \cite{Bu06}.

In this note we are interested in the 
dynamics of the restriction of the Teichm\"uller flow
to a compact invariant set.
For the formulation of our first result,
a continuous flow $\Phi^t$ on a compact metric space
$(X,d)$ is called \emph{expansive} if there is a constant $\delta >0$
with the following property. Let $x\in X$ and let $s:\mathbb{R}\to
\mathbb{R}$ be any continuous function with $s(0)=0$ and
$d(\Phi^t(x),\Phi^{s(t)}(x)) <\delta$ for all $t$. If $y\in X$ is
such that $d(\Phi^t(x),\Phi^{s(t)}(y))<\delta$ for all $t$ then
$y=\Phi^\tau(x)$ for some $\tau\in \mathbb{R}$ \cite{HK95}. 
Note that this definition of expansiveness does not
depend on the choice of the metric $d$ defining the
topology on $X$ so it makes sense to talk about an 
expansive flow on a compact metrizable space.

Let $\Gamma<{\rm Mod}(S)$ be a torsion free normal subgroup of
${\rm Mod}(S)$ of finite index. For example, the subgroup of all
elements which act trivially on 
$H_1(S,\mathbb{Z}/3\mathbb{Z})$ has this property.
Define 
$\hat {\cal Q}(S)={\cal Q}^1(S)/\Gamma.$
We show 

\begin{theo}\label{thm1}
The restriction of the Teichm\"uller flow to every
compact invariant subset $K$ of $\hat {\cal Q}(S)$ is
expansive.
\end{theo}

Periodic orbits of expansive flows on compact spaces 
are separated, 
and their asymptotic growth rate can be 
related to the \emph{topological entropy} of the flow.

For a compact $\Phi^t$-invariant subset $K$ of 
$\hat {\cal Q}(S)$ let
$h_{\rm top}(K)$ be the topological entropy of the restriction
of the Teichm\"uller flow to $K$. 
For any subset $U$ of $\hat{\cal Q}(S)$ (or of ${\cal Q}^1(S)/{\rm Mod}(S)$)
and for a number $r>0$
define $n_U(r)$ (or $n_U^\cap(r)$) to be the
cardinality of the set of all periodic orbits for $\Phi^t$
of period at most $r$ which are
entirely contained in $U$ (or which intersect $U$). 
Clearly $n_U^\cap(r)\geq n_U(r)$ for all $r$. We show

\begin{theo}\label{thm3}
Let $K\subset \hat {\cal Q}(S)$ be a compact
$\Phi^t$-invariant topologically transitive set. Then for
every open neighborhood $U$ of $K$ we have
\[\lim\sup_{r\to \infty}\frac{1}{r}\log n_K(r)\leq
h_{\rm top}(K)\leq \lim\inf_{r\to \infty}\frac{1}{r}\log n_U(r).\]
\end{theo}

It is not hard to see that 
Theorem \ref{thm1} and Theorem \ref{thm3} are equally valid
for compact invariant subsets of ${\cal Q}^1(S)/{\rm Mod}(S)$. However 
we did not find an argument which avoids using some 
differential geometric properties for the action of the mapping
class group on Teichm\"uller space which is not in the
spirit of this paper, so we omit a proof.

By the variational principle, the topological entropy of
a flow $\Phi^t$ on a compact space $K$ equals the supremum of the
metric entropies of all $\Phi^t$-invariant Borel probability
measures on $K$. Bufetov and Gurevich \cite{BG07} showed
that the supremum of the topological entropies of the
restriction of the Teichm\"uller flow to the moduli
space of \emph{abelian} differentials 
is just the metric entropy of the invariant probability
measure in the Lebesgue measure class, moreover this
Lebesgue measure is the unique measure of maximal entropy.
 
The following counting result implies that the entropy
$h$ of the $\Phi^t$-invariant Lebesgue measure on 
${\cal Q}^1(S)/{\rm Mod}(S)$
equals the supremum 
of the topological
entropies of the restrictions of the Teichm\"uller flow
to compact invariant sets.

\begin{theo}\label{thm2}
\begin{enumerate}
\item $\lim_{r\to \infty} \frac{1}{r}\log n^\cap_K(r)\leq h$
for every compact subset $K$ of \\
${\cal Q}^1(S)/{\rm Mod}(S)$.
\item For every $\epsilon >0$ there is a 
compact subset $K\subset {\cal Q}^1(S)/{\rm Mod}(S)$
(or $K\subset \hat{\cal Q}(S)$) such that 
\[\lim\inf_{r\to \infty} \frac{1}{r}\log n_K(r)\geq h-\epsilon.\]
\end{enumerate}
\end{theo}

The first part of Theorem \ref{thm2} is immediate from the results
of Eskin and Mirzakhani, however the proof given here is
very short and easy. The lower bound for the growth of
periodic orbits which remain in a fixed compact set
is the technically most involved part of this work.

The main tool we use for the proofs
of the above resuls 
is the \emph{curve graph} ${\cal C}(S)$ 
of $S$ and the relation between
its geometry and the geometry of Teichm\"uller space.
In Section 2 we introduce the curve graph, and
we summarize some results from \cite{H10} in the form
used in the later sections.
In Section 3 we  
investigate the Teichm\"uller flow $\Phi^t$ on
$\hat{\cal Q}(S)$ 
and we show Theorem \ref{thm1}.
In Section 4 we use the results from Section 3 to
establish a version of the Anosov
closing lemma for the restriction of the
Teichm\"uller flow to compact invariant subsets of
$\hat{\cal Q}(S)$ and show Theorem \ref{thm3}. 
The proof of Theorem \ref{thm2} 
is contained in Section 5.

\section{The curve graph and its boundary}

Let $S$ be an oriented surface of genus $g\geq 0$ with 
$m\geq 0$ punctures and $3g-3+m\geq 2$. 
The \emph{curve graph} ${\mathcal C}(S)$ of $S$ is the
graph whose vertices are the 
free homotopy classes of \emph{essential}
simple closed curves on $S$, i.e. simple closed
curves which are neither contractible
nor freely homotopic into a puncture. Two such curves
are joined by an edge if and only if they
can be realized disjointly. 
Since $3g-3+m\geq 2$ by assumption,
${\mathcal C}(S)$ is connected (see \cite{MM99}).
However, the curve graph is locally infinite. 
In the sequel we often do not distinguish between
a simple closed curve on $S$ and its free homotopy class.
Also, if we write $\alpha\in {\cal C}(S)$ then
we always mean that $\alpha$ is an essential simple closed curve,
i.e. $\alpha$ is a vertex in the curve graph ${\cal C}(S)$.

Providing each edge in ${\mathcal C}(S)$ with the standard
euclidean metric of diameter 1 equips the curve graph
with the structure of a geodesic metric space.
Since ${\mathcal C}(S)$ is not locally finite,
this metric space $({\mathcal C}(S),d)$
is not locally compact. Masur
and Minsky \cite{MM99} showed that nevertheless its geometry
can be understood quite explicitly. Namely, ${\mathcal C}(S)$
is hyperbolic of infinite diameter.
The mapping class group ${\rm Mod}(S)$ 
naturally acts
on ${\cal C}(S)$ as a group of simplicial isometries. 
A mapping class $\phi\in {\rm Mod}(S)$ is \emph{pseudo-Anosov}
if the cyclic subgroup of ${\rm Mod}(S)$ generated by
$\phi$ acts on the curve graph ${\cal C}(S)$ with unbounded orbits.

A \emph{geodesic lamination} for a complete
hyperbolic structure on $S$ of finite volume is
a \emph{compact} subset of $S$ which is foliated into simple
geodesics.
A geodesic lamination $\lambda$ on $S$ is called \emph{minimal}
if each of its half-leaves is dense in $\lambda$. Thus a simple
closed geodesic is a minimal geodesic lamination. A minimal
geodesic lamination with more than one leaf has uncountably
many leaves and is called \emph{minimal arational}.
A geodesic lamination $\lambda$ is said to \emph{fill up $S$} if
every simple closed geodesic on $S$ intersects $\lambda$
transversely. This is equivalent to stating that 
the complementary components of $\lambda$ are all 
topological discs
or once punctured topological discs.

A \emph{measured geodesic lamination} is a geodesic lamination
$\lambda$ together with a translation invariant transverse
measure. Such a measure assigns a positive weight to each compact
arc in $S$ which intersects $\lambda$ nontrivially and
transversely and whose
endpoints are contained in complementary regions of
$\lambda$.
The geodesic lamination $\lambda$ is called the
\emph{support} of the measured geodesic lamination; it consists of
a disjoint union of minimal components. Vice versa, every
minimal geodesic lamination is the support of a measured
geodesic lamination.

The space ${\cal M\cal L} $ of measured geodesic laminations
on $S$ can be equipped with the weak$^*$-topology.
Its projectivization ${\cal P\cal M\cal L}$ is called
the space of \emph{projective measured geodesic laminations}
and is homeomorphic to the sphere $S^{6g-7+2m}$.
There is a continuous symmetric pairing $\iota:{\cal M\cal L}\times
{\cal M\cal L}\to (0,\infty)$, the so-called
\emph{intersection form}, which satisfies
$\iota(a\xi,b\eta)=ab\iota(\xi,\eta)$ for all $a,b\geq 0$ and
all $\xi,\eta\in {\cal M\cal L}$.
By the Hubbard Masur theorem (see \cite{Hu06}),
for every $x\in {\cal T}(S)$ the space
${\cal P\cal M\cal L}$ of projective measured geodesic laminations
can naturally be identified with the projectivized cotangent space
of ${\cal T}(S)$ at $x$. Moreover, a quadratic differential
$q\in {\cal Q}^1(S)$ can be viewed as a pair
$(\lambda,\nu)\in {\cal M\cal L}\times {\cal M\cal L}$
with $i(\lambda,\nu)=1$ and the additional property
that $i(\lambda,\zeta)+i(\nu,\zeta)>0$ for all $\zeta\in {\cal M\cal L}$.

Since ${\cal C}(S)$ is a hyperbolic geodesic metric space, it
admits a \emph{Gromov boundary} $\partial {\cal C}(S)$
which is
a (non-compact) metrizable topological space equipped with
an action of ${\rm Mod}(S)$ by homeomorphisms
(see \cite{BH99} for the definition of 
the Gromov boundary of a hyperbolic geodesic 
metric space and for references).  
Following Klarreich \cite{Kl99} (see also \cite{H06}), this
boundary can naturally be identified with the space of 
all (unmeasured)
minimal geodesic laminations which fill up $S$, equipped with
the topology which is induced from the weak$^*$-topology on 
${\cal P\cal M\cal L}$ via the measure forgetting map.

Now let ${\cal F\cal M\cal L}\subset {\cal P\cal M\cal L}$ be the
${\rm Mod}(S)$-invariant Borel subset of all projective
measured geodesic laminations whose support is minimal and
fills up $S$. The discussion in the previous paragraph shows that
there is a continuous ${\rm Mod}(S)$-invariant surjection
\begin{equation}\label{F}
F:{\cal F\cal M\cal L}\to \partial {\cal C}(S)\end{equation}
which associates to a projective measured geodesic
lamination in ${\cal F\cal M\cal L}$ its support.

Since the curve graph is a hyperbolic geodesic metric
space, for every $c\in {\cal C}(S)$ there is a 
\emph{visual metric}
$\delta_c$ of uniformly bounded diameter
on the Gromov boundary $\partial{\cal
C}(S)$ of ${\cal C}(S)$  
(we refer to Chapter III.H of 
\cite{BH99} for details of this construction and for
references). These distances are related to the 
intrinsic geometry of ${\cal C}(S)$ as follows.

For a point $c\in {\cal C}(S)$, 
the \emph{Gromov product} at $c$ associates to 
points $x,y\in {\cal C}(S)$ the value
\begin{equation}\label{gromovprod}
(x\vert y)_c=\frac{1}{2}(d(x,c)+d(y,c)-d(x,y)). \notag
\end{equation}
The Gromov product can be extended to a Gromov
product $(\,\vert\,)_c$ for pairs of distinct points
in $\partial {\cal C}(S)$ by defining
\begin{equation}\label{prodatinf}
(\xi\vert \zeta)_c=\sup\liminf_{i,j\to \infty}(x_i\vert y_j)_c
\end{equation}
where the supremum is taken over all sequences $(x_i)$ and
$(y_j)$ in ${\cal C}(S)$ such that
$\xi=\lim x_i$ and $\zeta=\lim y_j$.
There are
numbers $\beta >0,\nu \in (0,1)$ such that 
\begin{align}\label{distancecompare1}
\nu e^{-\beta (\xi\vert \zeta)_c}& \leq \delta_c(\xi,\zeta)\leq
e^{-\beta (\xi\vert \zeta)_c}\text{ for all }\xi,\zeta\in 
\partial{\cal C}(S)
\text{ and }\\
\delta_c & \leq e^{\beta
d(c,a)}\delta_a \text{ for all }c,a\in {\cal C}(S). \notag
\end{align} 
The distances $\delta_c$ are equivariant with respect to
the action of ${\rm Mod}(S)$ on ${\cal C}(S)$ and
on $\partial {\cal C}(S)$. 
For $c\in {\cal C}(S)$, $\xi\in \partial{\cal C}(S)$
and $r>0$ denote by
$D_c(\xi,r)\subset \partial{\cal C}(S)$ 
the ball of radius $r$ about $\xi$ with respect to
the distance function $\delta_c$.

We will need a 
more precise quantitative relation between the distance
functions $\delta_c$ $(c\in {\cal C}(S))$. 
Even though this property is well known,
we did not find an explicit reference in the literature
and we include a sketch of a proof.

For a formulation, 
for a number $m>1$, an \emph{$m$-quasi-geodesic} in 
a metric space $(X,d)$ is a map 
$\gamma:J\to X$ such that
\begin{equation}\label{quasigeo}
\vert s-t\vert/m-m\leq d(\gamma(s),\gamma(t))\leq 
m\vert s-t\vert +m\,\text{ for all } s,t\in J\end{equation}
where $J\subset \mathbb{R}$ is a closed 
connected set.  Since ${\cal C}(S)$
is hyperbolic, every
quasi-geodesic ray $\gamma:[0,\infty)\to {\cal C}(S)$
converges as $t\to \infty$ in ${\cal C}(S)\cup \partial{\cal C}(S)$ to
an endpoint $\gamma(\infty)\in \partial{\cal C}(S)$. 

\begin{lemma}\label{expansion}
For every $m>1$ there are constants 
$a(m)>1,b(m)>0,\alpha_0(m)>0$ with the
following property. Let $\gamma:[0,\infty)\to {\cal C}(S)$
be an $m$-quasi-geodesic ray with endpoint $\gamma(\infty)\in 
\partial{\cal C}(S)$. Then for all $t\geq0$ we have
\[\delta_{\gamma(0)}\leq a(m)e^{-b(m)t}\delta_{\gamma(t)}
\text{ on }D_{\gamma(t)}(\gamma(\infty),\alpha_0(m)).\]
\end{lemma}
\begin{proof}
Since ${\cal C}(S)$ is a hyperbolic geodesic metric space, 
there is a constant $p>1$ depending on the hyperbolicity constant
such that every point $c\in {\cal C}(S)$
can be connected to every point $\xi\in \partial{\cal C}(S)$ by
a $p$-quasi-geodesic (for the particular case of the curve
graph see \cite{Kl99,H06,H10}).
Similarly, any two points $\xi\not=\zeta\in \partial{\cal C}(S)$
can be joined by a $p$-quasi-geodesic.

Let $m\geq p$. 
By hyperbolicity, there is a number $r(m)>0$ 
and for every $m$-quasi-geodesic 
triangle $T$ in ${\cal C}(S)$ 
with vertices $c\in {\cal C}(S),\xi \not=\zeta\in 
\partial{\cal C}(S)$ there is a point $u \in {\cal C}(S)$
whose distance to each of the sides of $T$  
is at most $r(m)$. The (non-unique) point $u$ is called
a \emph{center} of $T$. We claim that there
is a number $\chi(m)>0$ only depending on $m$ and the
hyperbolicity constant such that
\begin{equation}\label{center}
\delta_c(\xi,\zeta)\in [\chi(m) e^{-\beta d(c,u)},
e^{-\beta d(c,u)}/\chi(m)].
\end{equation}

Namely, let $\gamma_1,\gamma_2:[0,\infty)\to 
{\cal C}(S)$ be $m$-quasi-geodesic rays
which connect $c$ to $\xi,\zeta$. 
There is a universal constant $b>0$ only depending on the
hyperbolicity constant for ${\cal C}(S)$ 
such that
\begin{equation}\label{quasigromov}
(\gamma_1(\infty) \vert\gamma_2(\infty))_c-b\leq 
\liminf_{s,t\to \infty} (\gamma_1(s)\vert \gamma_2(t))_c
\leq (\gamma_1(\infty)\vert \gamma_2(\infty))_c
\end{equation}
(Remark 3.17.5 in Chapter III.H of \cite{BH99}).

By hyperbolicity, for sufficiently large $s,t$ we have
\[\vert d(c,\gamma_1(s))+d(c,\gamma_2(t))-
d(\gamma_1(s),\gamma_2(t))-2d(u,c)\vert \leq a\]
where $a>0$ is a constant only depending on the 
hyperbolicity constant for ${\cal C}(S)$ and on $m$.
Together with (\ref{prodatinf},\ref{distancecompare1},\ref{quasigromov}) 
this shows the estimate (\ref{center}).

Now let $\gamma:[0,\infty)\to {\cal C}(S)$ be any
$m$-quasi-geodesic ray with endpoint 
$\gamma(\infty)=\xi\in \partial{\cal C}(S)$,
let $\zeta\not=\xi\in \partial{\cal C}(S)$ and let $T$ 
be an $m$-quasi-geodesic triangle with side $\gamma$ and vertices
$\gamma(0),\xi,\zeta$. 
If $u\in {\cal C}(S)$ is a center for $T$ and if 
$\sigma\geq 0$ is 
such that $d(u,\gamma(\sigma))\leq r(m)$, then for every $s\in [0,\sigma]$
the distance between $u$ and a center for any $m$-quasi-geodesic
triangle with vertices $\gamma(s),\xi,\zeta$ is bounded from
above by a constant only depending on $m$ and the
hyperbolicity constant for ${\cal C}(S)$. 
In particular, by the
above discussion and the properties of an
$m$-quasi-geodesic, there are constants $a(m)>0,b(m)>0$ such that 
\begin{equation}
\delta_{\gamma(0)}(\xi,\zeta)
\leq a(m)e^{-b(m)s}\delta_{\gamma(s)}(\xi,\zeta)
\text{ for every }s\in [0,\sigma].
\end{equation}
From this the lemma follows.
\end{proof}

By a result of Bers, there is a constant $\chi_0=\chi_0(S)>0$
such that for every complete hyperbolic metric $x$ on $S$ of 
finite volume there is a \emph{pants decomposition}
for $S$ consisting of simple closed geodesics
of $x$-length at most $\chi_0$.
Define a map 
\begin{equation}\label{upsilon}
\Upsilon_{\cal T}:{\cal T}(S)\to {\cal C}(S)
\end{equation}
by associating
to a marked hyperbolic metric $x\in {\cal T}(S)$
a simple closed curve $\Upsilon_{\cal T}(x)$ whose
$x$-length $\ell_x(\Upsilon_{\cal T}(x))$ does not exceed $\chi_0$. 
There are choices involved in the definition of $\Upsilon_{\cal T}(x)$,
but for any two such choices and any $x$ the distance between
the images of $x$ is uniformly bounded. 

Denote by $d_T$ the distance on ${\cal T}(S)$ defined
by the Teichm\"uller metric.
Lemma 2.2 of \cite{H10}
shows that there is a number $L>1$ such that
\begin{equation}\label{upsilonlipschitz}
d(\Upsilon_{\cal T}(x),\Upsilon_{\cal T}(y))\leq Ld_{\cal T}(x,y)+L\,
\text{ for all }x,y\in {\cal T}(S).
\end{equation}

Choose a smooth function
$\sigma:[0,\infty)\to [0,1]$ with $\sigma[0,\chi_0]\equiv 1$ and
$\sigma[2\chi_0,\infty)\equiv 0$. 
For every $x\in {\cal T}(S)$ we obtain a finite Borel
measure $\mu_x$ on ${\cal C}(S)$ by defining
\begin{equation}\label{mu}
\mu_x=\sum_\beta \sigma(\ell_x(\beta))\delta_\beta
\end{equation}
where $\delta_\beta$ denotes the Dirac mass at $\beta$.
The total mass of $\mu_x$ is bounded from
above and below by a universal positive constant, and
the diameter of the support of $\mu_x$ in ${\cal C}(S)$
is uniformly bounded as well. The measures
$\mu_x$ are equivariant with respect to the action of 
the mapping class group on ${\cal T}(S)$ and
${\cal C}(S)$, and they 
depend continuously on $x\in {\cal T}(S)$
in the weak$^*$-topology. This means that for every
bounded function $f:{\cal C}(S)\to \mathbb{R}$
the function $x\to \int f d\mu_x$ is continuous.

For $x\in {\cal
T}(S)$ define a distance $\delta_x$ on $\partial {\cal C}(S)$ by
\begin{equation}\label{distance}
\delta_x(\xi,\zeta)=\int \delta_c(\xi,\zeta)
d\mu_x(c). \end{equation}
Clearly the metrics
$\delta_x$ are equivariant with respect to the
action of ${\rm Mod}(S)$ on ${\cal T}(S)$ and
$\partial{\cal C}(S)$. Moreover, there is a constant
$\kappa >0$ such that
\begin{equation}\label{deltacompare}
\delta_x\leq e^{\kappa d_T(x,y)}\delta_y\text{ and }
\kappa^{-1}\delta_x\leq
\delta_{\Upsilon_{\cal T}(x)}\leq \kappa\delta_x
\end{equation}
for all $x,y\in {\cal T}(S)$ (see p.230 and p.231 of \cite{H09}).

The main theorem of \cite{H10} relates the geometry
of Teichm\"uller space to the geometry of the curve
graph via the map $\Upsilon_{\cal T}$. For its formulation,
denote for $\epsilon >0$ by ${\cal T}(S)_\epsilon$ the 
set of all hyperbolic metrics whose \emph{systole} 
(i.e. the shortest length of a closed geodesic) is at least $\epsilon$.
For sufficiently small $\epsilon$ the set ${\cal T}(S)_\epsilon$ is connected,
and the mapping class group acts properly and cocompactly on 
${\cal T}(S)_\epsilon$. 

\begin{theorem}\label{teichcurvecomp}
\begin{enumerate}
\item For every $L>1$ there is a number $\epsilon=\epsilon(L)>0$
with the following property. Let $J\subset \mathbb{R}$
be a closed connected set of length at least $1/\epsilon$ and let
$\gamma:J\to {\cal T}(S)$ be an $L$-quasi-geodesic.
If $\Upsilon_{\cal T}\circ \gamma$
is an $L$-quasi-geodesic in ${\cal C}(S)$ then 
there is a Teichm\"uller geodesic
$\xi:J^\prime\to {\cal T}(S)_\epsilon$ such that
the Hausdorff distance between 
$\gamma(J)$ and $\xi(J^\prime)$ is at most $1/\epsilon$.
\item For every $\epsilon >0$ there is a number
$L(\epsilon)>1$ with the following property.
Let $J\subset \mathbb{R}$ be a closed connected
set and let $\gamma:J\to {\cal T}(S)$ be a
$1/\epsilon$-quasi-geodesic. If
there is a Teichm\"uller geodesic arc
$\xi:J^\prime\to {\cal T}(S)_\epsilon$ such that the
Hausdorff distance between 
$\gamma(J)$ and $\xi(J^\prime)$ is at most $1/\epsilon$ then
$\Upsilon_{\cal T}\circ \gamma$ is an $L(\epsilon)$-quasi-geodesic
in ${\cal C}(S)$.
\end{enumerate}
\end{theorem} 

Let ${\cal Q}^1(S)$ be the bundle of area
one quadratic differentials over Teichm\"uller space
${\cal T}(S)$ for $S$.
The mapping class group ${\rm Mod}(S)$ acts
properly discontinuously on ${\cal Q}^1(S)$ as
a group of bundle automorphisms. This action commutes with
the action of the Teichm\"uller geodesic flow $\Phi^t$.

An area one quadratic differential $q\in {\cal Q}^1(S)$
is determined by a pair $(q_v,q_h)$ of measured 
geodesic laminations, the  
\emph{vertical} and the 
\emph{horizontal} measured geodesic lamination of $q$, respectively.
For every $t\in \mathbb{R}$ the pair $(e^tq_v,e^{-t}q_h)$ corresponds
to the quadratic differential $\Phi^tq$.
The \emph{strong stable
manifold} 
\[W^{ss}(q)\subset {\cal Q}^1(S)\] of $q$ 
is defined as the set
of all quadratic differentials of area one whose vertical
measured geodesic lamination coincides \emph{precisely} with the
vertical measured geodesic lamination of $q$. The \emph{strong
unstable} manifold 
\[W^{su}(q)\subset {\cal Q}^1(S)\] is 
the set of all quadratic differentials of area one whose
horizontal measured geodesic lamination coincides
precisely with the horizontal measured geodesic lamination
of $q$. Define moreover the \emph{stable manifold} $W^s(q)$ of 
$q$ and the \emph{unstable manifold} $W^u(q)$ of $q$ by
$W^s(q)=\cup_{t\in \mathbb{R}}\Phi^tW^{ss}(q)$ and 
$W^u(q)=\cup_{t\in \mathbb{R}}\Phi^tW^{su}(q)$. 
As $q$ varies through ${\cal Q}^1(S)$,
the manifolds $W^{ss}(q), W^{su}(q),W^s(q),W^u(q)$
define continuous foliations of ${\cal Q}^1(S)$ which are 
called the \emph{strong stable}, the
\emph{strong unstable}, the \emph{stable} and the
\emph{unstable foliation}. These foliations 
are invariant under the action of ${\rm Mod}(S)$ and under the
action of the Teichm\"uller geodesic flow $\Phi^t$.

For every $q\in {\cal Q}^1(S)$ the map which associates 
to a quadratic differential its vertical
measured geodesic
lamination restricts to a homeomorphism of the 
unstable manifold $W^u(q)$ containing $q$ onto an open
dense subset of ${\cal M\cal L}$. The space ${\cal M\cal L}$
admits a natural ${\rm Mod}(S)$-invariant  
measure in the Lebesgue measure class which 
lifts to a locally finite measure on $W^u(q)$ in the Lebesgue 
measure class. The induced family of conditional measures 
$\tilde \lambda_q$ $(q\in {\cal Q}^1(S))$ on 
strong unstable manifolds transform under
the Teichm\"uller flow by $d\tilde \lambda_{\Phi^t q}\circ \Phi^t=
e^{ht}d\tilde \lambda_q$ where as before, $h=6g-6+2m$.
The measures $\tilde \lambda_q$ are equivariant under the
action of the mapping class group. Moreover, they are conditional
measures for a ${\rm Mod}(S)$-invariant locally finite measure
$\tilde \lambda$ on ${\cal Q}^1(S)$ in the Lebesgue measure class. 
The measure $\tilde \lambda$ is the lift of a \emph{finite} 
measure $\lambda$ 
on ${\cal Q}^1(S)/{\rm Mod}(S)$ which is $\Phi^t$-invariant and 
mixing.
The measure $\lambda$ gives full measure to the set of all 
quadratic differentials whose horizontal and vertical measured
geodesic laminations are uniquely ergodic and fill up $S$
(see \cite{M82,V86} for details). Moreover, by the Poincar\'e 
recurrence theorem, $\lambda$-almost every 
$q\in {\cal Q}^1(S)/{\rm Mod}(S)$ is \emph{recurrent},
i.e. it is contained in the $\omega$-limit set of its own orbit
under the Teichm\"uller flow.

Let
\begin{equation}\label{pi}
\pi:{\cal Q}^1(S)\to {\cal P\cal M\cal L}
\end{equation}
be the map which associates to a quadratic differential
its vertical projective measured geodesic
lamination. For every $q\in {\cal Q}^1(S)$
the restriction of the projection 
$\pi$ to $W^{su}(q)$ is a homeomorphism
of $W^{su}(q)$ onto the open
subset of ${\cal P\cal M\cal L}$ of all projective
measured geodesic laminations $\mu$ which
together with $\pi(-q)$ \emph{jointly fill up $S$},
i.e. are such that for every measured geodesic
lamination $\eta\in {\cal M\cal L}$ we have
$i(\mu,\eta)+i(\pi(-q),\eta)\not=0$ (note that this
makes sense even though the intersection form $i$ is
defined on ${\cal M\cal L}$ rather than on ${\cal P\cal M\cal L}$).
The measure class of the 
push-forward under $\pi$ of the measure $\tilde \lambda_q$ on
$W^{su}(q)$ does not depend on $q$ and defines a
${\rm Mod}(S)$-invariant ergodic measure class on 
${\cal P\cal M\cal L}$.

\section{Compact invariant sets are expansive}

Let again ${\cal Q}^1(S)$ be the bundle of area one quadratic
differentials over Teichm\"uller space ${\cal T}(S)$ of an 
oriented surface of genus $g\geq 0$ with $m\geq 0$
punctures and where $3g-3+m\geq 2$.
The mapping class group 
${\rm Mod}(S)$ acts on  ${\cal Q}^1(S)$, but this action is not
free and the quotient space ${\cal Q}^1(S)/
{\rm Mod}(S)$ is a non-compact orbifold rather than a manifold.

To overcome this
(mainly technical) difficulty we choose
a torsion free normal subgroup $\Gamma$ of ${\rm Mod}(S)$ of
finite index. For example, the group of all elements which act
trivially on $H_1(S,\mathbb{Z}/3\mathbb{Z})$ has this property.
Define
$\hat {\cal Q}(S)={\cal Q}^1(S)/\Gamma$
and let \begin{equation}
\Pi:{\cal
Q}^1(S)\to \hat{\cal Q}(S) \end{equation}
be the canonical projection. Since
the action of $\Gamma$ on ${\cal Q}^1(S)$ is free,
the map $\Pi$ is a covering. The Teichm\"uller flow
$\Phi^t$ acts on $\hat{\cal Q}(S)$. 

The goal of this section is to show 

\begin{theorem}\label{expansive}
The restriction of the Teichm\"uller flow to every
compact invariant subset $K$ of $\hat{\cal Q}(S)$ is
expansive.
\end{theorem}

We begin with the construction of a convenient metric
on ${\cal Q}^1(S)$ and $\hat {\cal Q}(S)$ inducing the usual topology. 
To this end, call a distance $d$ 
on a space $X$ a \emph{length metric} if the
distance between any two points is the infimum of the
lengths of all paths connecting these points.
In the sequel denote by
\begin{equation}
P:{\cal Q}^1(S)\to {\cal T}(S)
\end{equation}
the canonical projection.

\begin{lemma}\label{distancedef}
There is a complete ${\rm Mod}(S)$-invariant length
metric $d$ on ${\cal Q}^1(S)$
with the following properties.
\begin{enumerate}
\item The metric $d$ induces the usual topology.
\item The canonical projection $P:({\cal Q}^1(S),d)\to {\cal T}(S)$ is
distance non-decreasing where ${\cal T}(S)$ is equipped with
the Teichm\"uller metric.
\item Every orbit 
of the Teichm\"uller flow with its natural parametrization
is a minimal geodesic for $d$ parametrized by arc length.
\end{enumerate}
\end{lemma}
\begin{proof}
By the Hubbard Masur theorem (see \cite{Hu06} for a
presentation of this celebrated and by now classical result),
the restriction of 
the canonical projection 
$P:{\cal Q}^1(S)\to {\cal T}(S)$ to every unstable (or stable) manifold
in ${\cal Q}^1(S)$ 
is a homeomorphism onto ${\cal T}(S)$. Thus
the Teichm\"uller metric on ${\cal T}(S)$ lifts to 
a length metric on the leaves of the stable and of the
unstable foliation. 

Call a path
$\rho:[0,1]\to {\cal Q}^1(S)$ \emph{admissible} 
if there is a finite partition $0=t_0<\dots <t_k=1$ such that
the restriction of $\rho$ to each interval $[t_{i-1},t_i]$ is
entirely contained in a stable or in an unstable manifold.
For each such admissible path $\rho$ we can define its
\emph{length} to be the sum of the lengths 
with respect to the lifts of the Teichm\"uller metric
of the subsegments
of $\rho$ entirely contained in a stable or an unstable manifold.
For $q_0,q_1\in {\cal Q}^1(S)$ 
define $d(q_0,q_1)$ to be the infimum of the lengths
of all admissible paths connecting $q_0$ to $q_1$.
Then $d$ is a (a priori non-finite) distance function on
${\cal Q}^1(S)$ which satisfies the second
requirement in the lemma. 
The third property holds true since the projection to
${\cal T}(S)$ of an orbit of the Teichm\"uller flow
is a Teichm\"uller geodesic of the same length and hence
realizes the distance between its endpoints.
By the second property for $d$ and
the definition, the $d$-length of every path in ${\cal Q}^1(S)$
which is 
entirely contained in a stable or an unstable manifold coincides
with the length of its projection to ${\cal T}(S)$. As a consequence,
the metric $d$ is a length metric.

We are left with showing that $d$ induces the usual
topology on ${\cal Q}^1(S)$.
For this let $q\in {\cal Q}^1(S)$ and let $\epsilon >0$.
We have to show that the $\epsilon$-ball about $q$ for the
distance $d$ contains a neighborhood of $q$
in ${\cal Q}^1(S)$. For this denote
for $z\in {\cal Q}^1(S)$ and $r>0$ by
$B^i(q,r)$ 
the open $r$-ball about $q$ in $W^i(q)$ with respect
to the lift of the Teichm\"uller metric
$(i=s,u)$. For each 
$z\in B^s(q,\epsilon/2)$, the open ball $B^u(z,\epsilon/2)$
of radius $\epsilon/2$  
about $z$ in $W^u(z)$ is an open neighborhood of $z$
in $W^u(z)$ whose closure is compact and 
depends continuously on $z$
in the Hausdorff topology for compact subsets of 
${\cal Q}^1(S)$ by the 
Hubbard Masur theorem. Then 
$U=\cup_{z\in B^s(q,\epsilon/2)}B^u(z,\epsilon/2)$
is an open neighborhood of $q$ in ${\cal Q}^1(S)$. Moreover
by construction, $U$ is contained in the $\epsilon$-ball about
$q$ with respect to the distance function $d$.
This completes the proof of the lemma.
\end{proof}

The distance $d$ on ${\cal Q}^1(S)$ constructed in Lemma \ref{distancedef}
induces a metric
on $\hat{\cal Q}(S)$, again denote by $d$, via
\begin{equation}\label{distanceprojection}
d(x,y)=\inf\{d(\tilde x,\tilde y)\mid \Pi(\tilde x)=
\Pi(\tilde y)\}
\end{equation}
where as before,
$\Pi:{\cal Q}^1(S)\to \hat {\cal Q}(S)$ is the canonical projection.
In the sequel we always use these distance functions on 
${\cal Q}^1(S)$ and $\hat{\cal Q}(S)$ without further mentioning.

With these preparations, we can prove Theorem \ref{expansive}.

{\it Proof of Theorem \ref{expansive}.}
Let $K\subset \hat{\cal Q}(S)$ be a compact
$\Phi^t$-invariant subset and let $\tilde K\subset {\cal Q}^1(S)$ 
be the
preimage of $K$ under the canonical projection
$\Pi:{\cal Q}^1(S)\to \hat{\cal Q}(S)$.
Let as before $P:{\cal Q}^1(S)\to {\cal T}(S)$ be the canonical
projection.
By the second part of Theorem \ref{teichcurvecomp},
there is a constant $p>0$ only depending on $K$
such that for every
$q\in \tilde K$ the curve
$t\to \Upsilon_{\cal T}(P\Phi^tq)$ is a
$p$-quasi-geodesic in ${\cal C}(S)$,
i.e. we have
\begin{equation}\label{quasi}\vert t-s\vert/p-p\leq
d(\Upsilon_{\cal T}(P\Phi^tq),\Upsilon_{\cal T}(P\Phi^sq))\leq
p\vert t-s\vert +p\quad\text{for all}\quad s,t\in \mathbb{R}.
\end{equation}

Since by inequality (\ref{upsilonlipschitz}) the map
$\Upsilon_{\cal T}$ is coarsely Lipschitz, this shows that
lifts of orbits of $\Phi^t\vert K$ which are contained
in the same unstable manifold diverge linearly in forward
direction. We therefore just have to relate distances
in the curve graph to distances in ${\cal Q}^1(S)$ for the
metric $d$ in a quantitative way. 
The remainder of the argument
establishes such a control.

Let ${\cal F}:{\cal Q}^1(S)\to {\cal Q}^1(S)$ be the
\emph{flip} $q\to {\cal F}(q)=-q$. This flip is equivariant
with respect to the action of the mapping class group
and hence it descends to a continuous involution of $\hat {\cal Q}(S)$
which we denote again by ${\cal F}$.
Recall that the Gromov boundary $\partial {\cal C}(S)$ of ${\cal
C}(S)$ can be identified with the set of all (unmeasured) minimal
geodesic laminations on $S$ which fill up $S$.
Let again $\pi:{\cal Q}^1(S)\to {\cal P\cal M\cal L}$
be the canonical projection.
If $q\in \tilde K\cup {\cal F}(\tilde K)$ then the
support $F(\pi q)$ of the vertical measured 
geodesic lamination of $q$ is
\emph{uniquely ergodic}, which means that $F(\pi q)$ 
admits a unique transverse
measure up to scale, and $F(\pi q)$ is minimal and fills
up $S$ \cite{M82}. Moreover, the $p$-quasi-geodesic $t\to
\Upsilon_{\cal T} (P\Phi^tq)$ converges in ${\cal C}(S)\cup
\partial {\cal C}(S)$ to $F(\pi q)$ (the latter statement
follows from the explicit identification of 
$\partial {\cal C}(S)$ with the set of minimal geodesic laminations
which fill up $S$ established in \cite{Kl99,H06}, see also
\cite{H10}).

Write $A=\pi(\tilde K\cup {\cal F}(\tilde K))$. 
Then $A$ is a $\Gamma$-invariant Borel subset of 
${\cal P\cal M\cal L}$. By the consideration in the previous paragraph,
the restriction to $A$ of the 
map $F:{\cal F\cal M\cal L}\to \partial {\cal C}(S)$ introduced
in (\ref{F}) of Section 2 is a
$\Gamma$-equivariant continuous injection 
\[F_A:A\to\partial {\cal C}(S)\] which 
associates to a projective measured geodesic lamination 
contained in $A$
its support. In other words,
we can identify the set $A$ with a subset of $\partial {\cal C}(S)$.

Recall from equation (\ref{distance}) the definition of the
distances $\delta_x$ $(x\in {\cal T}(S))$ on 
the Gromov boundary $\partial{\cal C}(S)$ of ${\cal C}(S)$.
Since the map $F_A:A\to \partial{\cal C}(S)$ is injective,
for every $q\in {\cal Q}^1(S)$ the function
$(\xi,\zeta)\in A\times A\to \delta_{Pq}(F_A\xi,F_A\zeta)$ is a distance on $A$.
For simplicity of notation, we denote this distance again by $\delta_{Pq}$.
The topology on $A$ defined by this distance is just the
subspace topology of $A$ as a subset of ${\cal P\cal M\cal L}$
(this is the result of \cite{Kl99}).

For $q\in \tilde K\cup {\cal F}(\tilde K)$ 
denote by $D_q(\pi(q),r)\subset A$ the
ball of radius $r$ in $A$ about $\pi(q)$ with
respect to this distance. By continuity of 
the projection $\pi$ and the map $F_A$ and by the relation
(\ref{deltacompare}) between the distances
$\delta_x$ and $\delta_{y}$ for $x,y\in {\cal T}(S)$,
the ball 
$D_q(\pi(q),r)$ depends continuously on $q$ in the following sense.
For every $q\in \tilde K\cup {\cal F}(\tilde K)$, 
every $r>0$ and every 
$\epsilon \in (0,r)$ there is a neighborhood $U$ of $q$ in 
$\tilde K\cup {\cal F}(\tilde K)$
such that for every $u\in U$ we have 
\begin{equation}\label{continuousdependence}
D_u(\pi(u),r-\epsilon)\subset
D_q(\pi(q),r)\subset D_u(\pi(u),r+\epsilon).
\end{equation}

By inequalities (\ref{deltacompare}) and (\ref{quasi})
and by Lemma \ref{expansion},
there are numbers $\alpha>0,
a>1,b>0$ such that for every $q\in \tilde K$ and
for all $t>0$ we have
\begin{align}\label{contraction1}
\delta_{P\Phi^{-t}q}& \leq ae^{-bt}\delta_{Pq}\text{ on }
D_{q}(\pi(q),2\alpha)\text{ and }\\
\delta_{P\Phi^{t}q}& \leq ae^{-bt}\delta_{Pq}\text{ on }
D_{{\cal F}(q)}(\pi({\cal F}q),2\alpha). \notag
\end{align}

For $q\in {\cal Q}^1(S)$ and $\beta >0$ denote by
$B(q,\beta)$ the ball of radius $\beta$ 
about $q$ in ${\cal Q}^1(S)$ with respect to 
the length metric $d$ defined in 
Lemma \ref{distancedef}. 
Since the projection $\pi$ is continuous,
for every $q\in \tilde K\cup {\cal F}(\tilde K)$ there is a number
$\epsilon(q)>0$ such that
$\pi (B(q,\epsilon(q))\cap (\tilde K\cup {\cal F}\tilde K))\subset 
D_q(\pi(q),\alpha)$ where 
$\alpha>0$ is as in the inequalities (\ref{contraction1}).
By the continuity properties (\ref{continuousdependence})
of the balls $D_u(\pi(u),r)$ for $u\in \tilde K\cup
{\cal F}(\tilde K)$ and $r>0$, by 
invariance under the action of the group $\Gamma<{\rm Mod}(S)$  
and cocompactness we can find a universal number $\beta_0 >0$ 
such that
\begin{equation}\label{contraction10}
\pi (B(q,\beta_0)\cap (\tilde K\cup {\cal F}\tilde K)) \subset
D_q(\pi(q),\alpha)\text{ for all } q\in \tilde K\cup {\cal F}(\tilde K).
\end{equation}

Denote again by $d$ the distance on $\hat{\cal Q}(S)$ induced
in equation (\ref{distanceprojection})
from the distance on ${\cal Q}^1(S)$. 
Since $K\cup {\cal F}(K)$ is compact and $\Pi:{\cal Q}^1(S)\to 
\hat{\cal Q}(S)$ is a covering, there is a 
number $\beta<\beta_0$ such that for every 
$q\in K\cup {\cal F}(K)$ and every lift $\tilde q$ of $q$ to 
${\cal Q}^1(S)$ 
the ball $B(q,\beta)$ in $\hat{\cal Q}(S)$ of radius $\beta$ 
about $q$ is 
the homeomorphic image under $\Pi$ of 
the ball $B(\tilde q,\beta)$.
Now the orbits of $\Phi^t$ are geodesics for the distance $d$. 
Hence 
if $x\in K$, if $s:\mathbb{R}\to \mathbb{R}$ is 
a continuous function such that
$s(0)=0$ and 
$d(\Phi^tq,\Phi^{s(t)}q)< \beta/2$ for all $t$ then 
by the choice of $\beta$ we have 
$\vert t-s(t)\vert < \beta/2$ for all $t$. 

This implies that for this function $s$,  
if $u\in K$ is such that 
$d(\Phi^{t}q,\Phi^{s(t)}u)<\beta/2$ for all $t$ then 
$d(\Phi^tu,\Phi^{s(t)}u)<\beta/2$ 
(since $\vert t-s(t)\vert <\beta/2$ and orbits of 
the Teichm\"uller flow are geodesics) 
and hence
$d(\Phi^tu,\Phi^tq)<\beta$ for all $t$ by the triangle inequality.
In particular, for a lift $\tilde q\in {\cal Q}^1(S)$
of $q$ and a lift $\tilde u$ of $u$ with 
$d(\tilde q,\tilde u)<\beta$ we have $d(\Phi^t \tilde q,
\Phi^t\tilde u)<\beta$ for all $t\in \mathbb{R}$.  

Let $W^s_{\rm loc}(q)$ be the connected component containing 
$q$ of the intersection
of $B(q,\beta)$ with the stable manifold $W^s(q)$ of $q$. 
We claim that $u\in W^{s}_{\rm loc}(q)$. 
For this assume otherwise. Let again $\tilde q$ be a preimage
of $q$ in ${\cal Q}^1(S)$
and let $\tilde u\in {\cal Q}^1(S)$ be the preimage of 
$u$ with $d(\tilde q,\tilde u)<\beta$.
If $\pi(\tilde q)\not=\pi(\tilde u)$ then
$\delta_{P\tilde q}(\pi (\tilde q),\pi(\tilde u))>0$ and   
by continuity, our choice of 
$\alpha$ and the estimates (\ref{contraction1}), 
there is a number $t>0$ such that $\delta_{P\Phi^{t}\tilde q}
(\pi(\tilde q),\pi(\tilde u))=\alpha$. On the other hand, 
for every $s\in [0,t]$ 
the distance in ${\cal Q}^1(S)$ 
between $\Phi^s\tilde q,\Phi^s\tilde u$
is smaller than $\beta$ which 
is a contradiction to the choice of $\beta<\beta_0$ and the relation 
(\ref{contraction10}).

In the same way we conclude that $u$ is contained
in the intersection of $B(q,\beta)$ with the 
local unstable manifold of $q$. Now the intersection of the 
local stable manifold $W_{\rm loc}^s(q)$ with the local
unstable manifold $W^u_{\rm loc}(q)$ is contained in the
orbit of $q$ under the Teichm\"uller flow $\Phi^t$ 
which completes the proof of Theorem \ref{expansive}.
\qed

For a compact $\Phi^t$-invariant subset $K$ of $\hat{\cal Q}(S)$
let $h_{\rm top}(K)$ be the topological entropy of the restriction
of $\Phi^t$ to $K$. For $r>0$ let moreover $n_K(r)$ be the number
of periodic orbits of $\Phi^t$ of period at most $r$ 
which are contained in $K$. 
Since by Theorem \ref{expansive} 
the restriction of the flow $\Phi^t$ to $K$ is expansive,
by Proposition 3.2.14 of \cite{HK95} we
have

\begin{corollary}\label{lowerbound}
Let $K\subset \hat{\cal Q}(S)$ be a compact $\Phi^t$-invariant
set; then
\[\lim\sup_{r\to \infty} \frac{1}{r}\log n_K(r)\leq
h_{\rm top}(K).\]
\end{corollary}

\section{An Anosov closing lemma}

The goal of this section is to establish 
a version of the Anosov
closing lemma for the restriction of the
Teichm\"uller flow to a compact invariant set
$K\subset \hat{\cal Q}(S)$. 
The classical Anosov closing lemma roughly states that 
for a hyperbolic flow on a closed Riemannian
manifold, a closed curve consisting of 
sufficiently long orbit segments which are connected at the
endpoints by sufficiently short arcs is closely fellow-traveled
by a periodic orbit.

We continue to use the assumptions and notations from
Section 2 and Section 3. In particular, we always use the 
distances on ${\cal Q}^1(S)$ and $\hat {\cal Q}(S)$ defined in 
Lemma \ref{distancedef} and in equation (\ref{distanceprojection}).
For a precise formulation of our version
of an Anosov closing lemma for the Teichm\"uller flow,
using the assumptions and notations from Section 2 and Section 3
we define.

\begin{definition}\label{pseudoorbit}
For $n>0,\epsilon
>0$, an \emph{$(n,\epsilon)$-pseudo-orbit} for the
Teichm\"uller flow $\Phi^t$ on $\hat{\cal Q}(S)$
consists of a sequence of points
$q_0,q_1,\dots,q_k\in \hat{\cal Q}(S)$ and a sequence of numbers
$t_0,\dots,t_{k-1}\in [n,\infty)$ with the following properties.
\begin{enumerate}
\item For every $j\leq k$ the $2\epsilon$-neighborhood
of $q_j$ is contained in a contractible subset of $\hat{\cal Q}(S)$.
\item For every $j<k$ we have
$d(\Phi^{t_j}q_j,q_{j+1})\leq \epsilon$.
\end{enumerate}
The pseudo-orbit is \emph{contained} in a compact set $K$
if for all $i$ and all $t\in [0,t_i]$ we have
$\Phi^tq_i\in K$.
The pseudo-orbit is called \emph{closed} if
$q_0=q_k$.
\end{definition}

An $(n,\epsilon)$-pseudo-orbit $q_0,\dots,q_k$ determines an
essentially unique arc
connecting $q_0$ to $q_k$
which we call a \emph{characteristic arc}. Namely, by
assumption, for each $j<k$ the $2\epsilon$-neighborhood
of $q_{j+1}$
is contained in a contractible subset of $\hat{\cal Q}(S)$. 
Hence the homotopy class with fixed endpoints
in $\hat {\cal Q}(S)$ 
of an arc of length smaller than $2\epsilon$ 
connecting $\Phi^{t_j}q_j$ to $q_{j+1}$ is unique.
We define a
characteristic arc of the $(n,\epsilon)$-pseudo-orbit to be
an arc connecting $q_0$ to $q_k$ which
is obtained by successively
joining the endpoint of the orbit segment $\{\Phi^tq_j\mid 0\leq
t\leq t_j\}$ to $q_{j+1}$ with an arc of length smaller than $2\epsilon$
which is parametrized on
the unit interval $(j=0,\dots,k-1)$. 
The points $q_i$ $(1\leq i\leq k)$ are
called the \emph{breakpoints} of the characteristic
arc of the pseudo-orbit.
The \emph{characteristic homotopy class} of the pseudo-orbit
is the homotopy class with fixed endpoints of a characteristic arc
connecting $q_0$ to $q_k$. Note that this is independent of the
choice of a characteristic arc.
If the pseudo-orbit is closed then it determines
a closed characteristic curve and hence a free
homotopy class of closed curves in $\hat{\cal Q}(S)$
which we call the \emph{characteristic free homotopy class} of 
the closed pseudo-orbit.

By abuse of notation, denote again by $P:\hat{\cal Q}(S)\to
{\cal T}(S)/\Gamma$ the canonical projection.

\begin{definition}\label{shadow}
An $(n,\epsilon)$-pseudo-orbit
$q_0,\dots,q_k$ as in Definition \ref{pseudoorbit} is
\emph{$\delta$-shadowed} by an orbit segment
$\zeta=\{\Phi^tq\mid t\in [0,\tau]\}$
for some $q\in \hat{\cal Q}(S)$ and some $\tau >0$ if the following
holds.
\begin{enumerate}
\item There is a number $\alpha\leq \delta$ such that
the $\alpha$-neighborhoods of $Pq_0,Pq_k$ 
in ${\cal T}(S)/\Gamma$ with respect to the projection
of the Teichm\"uller metric are contractible
and contain $Pq,P\Phi^\tau q$.
\item There is a lift $\tilde \zeta$ 
to ${\cal Q}^1(S)$ of the orbit segment $\zeta$ and
a lift $\tilde\gamma$ to ${\cal Q}^1(S)$ of a
characteristic arc $\gamma$ for the pseudo-orbit
with the following properties. The distance between the
endpoints of $P\tilde \gamma,P\tilde \zeta$
is at most $\alpha$, and the
Hausdorff distance 
between $\tilde\gamma$ and $\tilde \zeta$
is at most $\delta$.
\end{enumerate}
A closed pseudo-orbit in $\hat{\cal Q}(S)$ is
$\delta$-shadowed by a periodic orbit
if in addition to the above requirements
the orbit $\{\Phi^tq\mid t\in [0,\tau]\}$ is closed.
\end{definition}

Note that every point in the orbifold 
${\cal Q}^1(S)/{\rm Mod}(S)$ admits a contractible neighborhood,  so we can
use the above definition is the same way for the Teichm\"uller flow
on the orbifold ${\cal Q}^1(S)/{\rm Mod}(S)$.

Using Definition \ref{pseudoorbit} and Definition \ref{shadow}
we can now formulate a version
of the Anosov closing lemma for the Teichm\"uller flow which
is the main result of this section.

\begin{theorem}\label{Anosov}
For every compact $\Phi^t$-invariant
set $K\subset \hat{\cal Q}(S)$
there are numbers $\epsilon_1=\epsilon_1(K) >0$, 
$n=n(K)>0$, $b=b(K)>0$ such that every
$(n,\epsilon_1)$-pseudo-orbit contained in $K$ is
$b$-shadowed by an
orbit. Moreover, for every $\delta >0$ there is a number
$\epsilon_2=\epsilon_2(K,\delta)<\epsilon_1$ such that
a closed $(n,\epsilon_2)$-pseudo-orbit contained in $K$ is
$\delta$-shadowed by a periodic orbit which is 
contained in the characteristic free homotopy class of the
closed pseudo-orbit.
\end{theorem}

For the proof of Theorem \ref{Anosov} we need 
the following technical
preparation which is also used in Section 5. To this end, recall
that a point $q\in {\cal Q}^1(S)/{\rm Mod}(S)$ is recurrent if
$q$ is contained in the $\omega$-limit set of its own orbit
under the Teichm\"uller flow. For some $m>1$,
an \emph{unparametrized $m$-quasi-geodesic}
in ${\cal C}(S)$ is an arc $\gamma:J\to {\cal C}(S)$
with the property that  
there is an orientation preserving homeomorphism
$\phi:I\to J$ such that $\gamma\circ \phi:I\to {\cal C}(S)$ 
is an $m$-quasi-geodesic.
We show

\begin{lemma}\label{quasigeodesic}
\begin{enumerate}
\item
For every compact $\Phi^t$-invariant
set $K\subset \hat{\cal Q}(S)$ there are
numbers $\epsilon_0=\epsilon_0(K)>0, 
n_0=n_0(K)>0,\ell_0=\ell_0(K)>1$
depending on $K$
with the following property. Let 
$q_0,\dots,q_k\in \hat {\cal Q}(S)$
be an $(n_0,\epsilon_0)$-pseudo-orbit contained in $K$ and
let $\tilde\gamma$ be a lift to ${\cal Q}^1(S)$ of a characteristic
arc connecting $q_0$ to $q_k$. Then the arc
$t\to \Upsilon_{\cal T}(P\tilde\gamma(t))$ is an
$\ell_0$-quasi-geodesic in ${\cal C}(S)$.
\item There is a number $m>1$ and 
for every recurrent point $q\in {\cal Q}^1(S)/{\rm Mod}(S)$ there are
numbers $\epsilon_0(q)>0,n_0(q)>0$ with the following properties.
Let $q_0,\dots,q_k$
be an $(n_0(q),\epsilon_0(q))$-pseudo-orbit with 
$d(q_i,q)\leq \epsilon_0(q)$ for all $i$ 
and let $\tilde\gamma$ be a lift to ${\cal Q}^1(S)$ of a characteristic
arc $\gamma$ connecting $q_0$ to $q_k$. Then the arc
$t\to \Upsilon_{\cal T}(P\tilde\gamma(t))$ is an unparametrized
$m$-quasi-geodesic in ${\cal C}(S)$. 
\end{enumerate}
\end{lemma}
\begin{proof}
Recall from (\ref{distancecompare1}) and
(\ref{distance}) of Section 2 the
definition of the distance functions $\delta_c$ $(c\in {\cal C}(S))$
and $\delta_x$ $(x\in {\cal T}(S))$ 
on the Gromov boundary $\partial {\cal C}(S)$
of ${\cal C}(S)$. By hyperbolicity, every
quasi-geodesic ray $\zeta:[0,\infty)\to {\cal C}(S)$
converges as $t\to \infty$ to a point
$R(\zeta)\in \partial{\cal C}(S)$.
Moreover, for every $p >1$ there are numbers
$m(p)>0,\alpha(p)>0,\beta(p)>1$ with the following property.

Let $k>0$ and let
$\zeta_0,\dots,\zeta_{k-1}:\mathbb{R}\to {\cal C}(S)$ be
infinite $p$-quasi-geodesics. Let $L>1$ be as
in inequality (\ref{upsilonlipschitz}).
Suppose that
for every $j\leq k-2$ there is a number $T_j>m(p)$
such that $d(\zeta_j(T_j),\zeta_{j+1}(0))\leq 2L$.
For each $j\leq k-1$ let $\rho_j:[0,1]\to {\cal C}(S)$
be any map whose image is 
contained in the $2L$-neighborhood of
$\zeta_j(T_j)$ and let $\tilde \zeta_j$ be the composition
of $\zeta_j[0,T_j]$ with $\rho_j$ parametrized in the natural
way on $[0,T_j+1]$.
If $\delta_{\zeta_{j+1}(0)}(R(\zeta_j),R(\zeta_{j+1}))<\alpha(p)$ 
for all $j$ then
the curve $\zeta:[0,\sum_iT_i+k]\to {\cal C}(S)$ defined by
\begin{equation}
\zeta(t)=\tilde\zeta_j(t-\sum_{i=0}^{j-1}T_i-j)\text{ for }
t\in [\sum_{i=0}^{j-1}T_i +j,\sum_{i=0}^jT_{i}+j+1]
\end{equation}
is a $\beta(p)$-quasi-geodesic.

Let $K\subset \hat{\cal Q}(S)$ be a compact $\Phi^t$-invariant
set and let
$\tilde K\subset {\cal Q}^1(S)$ be the preimage of $K$ under the
natural projection. 
By the second part of Theorem \ref{teichcurvecomp} there
is a  number $p >1$ such that
for every $q\in \tilde K$ the assignment
$t\to \Upsilon_{\cal T}(P\Phi^tq)$ $(t\in \mathbb{R})$ is
a $p$-quasi-geodesic. 

For $q\in \tilde K$ we have
$F(\pi(q))=R(t\to \Upsilon_{\cal T}(P\Phi^tq))\in
\partial {\cal C}(S)$. 
Let $\kappa>0$ be as in 
(\ref{deltacompare}) of Section 2.
By continuity, for every $q\in \tilde K$
there is a number $\epsilon(q)>0$ such that
for every point $\tilde q\in \tilde K$ which is contained in the
$2\epsilon(q)$-neighborhood of $q$ the
$\delta_{Pq}$-distance between
$F(\pi(q))$ and $F(\pi(\tilde q))$
is smaller than $\alpha(p)/\kappa$ where
$\alpha(p)>0$ is as in the first paragraph of this proof.
By continuity,
invariance under the action of the mapping class group 
on ${\cal Q}^1(S)$ and $\partial{\cal C}(S)$ and
cocompactness of the action 
of $\Gamma$ on $\tilde K$, there is a number 
$\epsilon_0\in (0,1/2)$ which has
this property for all $q\in \tilde K$ (compare the proof of 
Theorem \ref{expansive} for a similar statement).

Let $m(p)>0$ be as in the first paragraph of this proof.
Let $\tilde \gamma$ be the
lift to ${\cal Q}^1(S)$ of a characteristic arc
of an $(m(p),\epsilon_0)$-pseudo-orbit
$q_0,q_1,\dots,q_k$ contained in 
$K \subset \hat{\cal Q}(S)$. Then
$\tilde \gamma$ is a composition of 
curves $\tilde\gamma_i$ $(i=1,\dots,k)$ 
where the curve $\tilde\gamma_i$ 
is a lift to ${\cal Q}^1(S)$ of an orbit segment
for the Teichm\"uller flow  
of length at least $m(p)$ beginning at $x_{i-1}$
and an arc of length at most 
$2\epsilon_0<1$ parametrized on $[0,1]$, with endpoint
$x_i$.
By the choice of $m(p)$, of $\epsilon_0>0$
and by the construction of the curves $\tilde\gamma_i$, the curve
$\Upsilon_{\cal T}(P\tilde \gamma)$ is of the form described in
the second paragraph of this proof and hence
it is a $\beta(p)$-quasi-geodesic
in ${\cal C}(S)$.
This shows the first part of the lemma.

The second part of the lemma follows in the same way. 
By \cite{MM99} there is a number
$\ell>0$ such that for every $\tilde q\in {\cal Q}^1(S)$ the 
map $t\to \Upsilon_{\cal T}(P\Phi^t\tilde q)$ is an unparametrized
$\ell$-quasi-geodesic (this number $\ell >0$ does not coincide
with the number $p>1$ above, see also
\cite{H10}). 
Let $m(\ell)>0,\alpha(\ell)>0,p(\ell)>1$
be as in the first paragraph of this proof.
Let $q\in {\cal Q}^1(S)/{\rm Mod}(S)$ 
be a recurrent point and let $\tilde q\in 
{\cal Q}^1(S)$ be a lift of $q$. Then the vertical measured
geodesic lamination of $\tilde q$ is uniquely ergodic and fills 
up $S$ \cite{M82}, and the unparametrized
$\ell$-quasi-geodesic $t\to \Upsilon_{\cal T}(P\Phi^t\tilde q)$ 
is of infinite diameter. By continuity
and by Lemma 2.4 of \cite{H10} 
there is a neighborhood $V$ of $\tilde q$ in 
${\cal Q}^1(S)$ and a number $T(q)>0$ 
such that
\begin{equation}
d(\Upsilon_{\cal T}(P\Phi^tu),\Upsilon_{\cal T}(Pu))\geq
\ell m(\ell)+\ell\text{ for all }u\in V\text{ and all }t\geq T(q).
\end{equation}
In particular, if $u\in V$ and if 
$\rho_u:[0,a)\to [0,\infty)$ $(a\in (0,\infty])$ is
a homeomorphism with the property that the map
$t\to \Upsilon_{\cal T}(P\Phi^{\rho_u(t)}u)$ is 
a \emph{parametrized} $\ell$-quasi-geodesic in ${\cal C}(S)$ then
$\rho_u(m(\ell))\leq T(q)$.

The subset ${\cal U}$ of ${\cal Q}^1(S)$ of all 
points with uniquely ergodic vertical measured
geodesic lamination which fills up $S$ is dense \cite{M82}. Each
$u\in {\cal U}$ defines a point
$F\pi(u) \in \partial{\cal C}(S)$ which is 
just the endpoint of the infinite unparametrized
$\ell$-quasi-geodesic $t\to \Upsilon_{\cal T}(P\Phi^tu)$.
The map $u\in {\cal U}\to F\pi(u)\in \partial{\cal C}(S)$
is continuous. Thus we can find a number 
$\epsilon_0(q)\in (0,1/2)$ which is small enough that 
the $2\epsilon_0(q)$-neighborhood $W$ of $\tilde q$ 
in ${\cal Q}^1(S)$ is contained in 
$V$ and that 
for every point $u\in W\cap {\cal U}$ the
$\delta_{P\tilde q}$-distance between
$F(\pi(\tilde q))$ and $F(\pi(u))$
is smaller than $\alpha(\ell)/\kappa$ where as before,
$\kappa >0$ is as in (\ref{deltacompare}) in Section 2.
The second part of the lemma holds true for the 
numbers $m=m(\ell)>0,\epsilon_0(q)>0,n_0(q)=T(q)>0$.
\end{proof}

{\it Proof of Theorem \ref{Anosov}.}
Let $K\subset \hat{\cal Q}(S)$
be any compact $\Phi^t$-invariant set and let
$\tilde K$ be the preimage of $K$ in ${\cal Q}^1(S)$
under the natural projection.
Let $\epsilon_0=\epsilon_0(K)\in (0,1/2),
n_0=n_0(K)>0$ be as in the first part of Lemma \ref{quasigeodesic}.
Let $q_0,\dots,q_k$ be an $(n_0,\epsilon_0)$-pseudo-orbit for
$\Phi^t$ which is contained in $K$
and let $t_0,\dots,t_{k-1}\in [n_0,\infty)$
be as in the definition of a pseudo-orbit
such that $d(\Phi^{t_i}q_i,q_{i+1})\leq \epsilon_0$
for $i<k$. Let $\gamma$ be a characteristic arc of this
pseudo-orbit which is parametrized on
$[0,\sum_{i=0}^{k-1}t_i+k]$ in such a
way that for each $j$ the restriction
of $\gamma$ to $[\sum_{i< j}t_i+j,\sum_{i<  j+1}t_i+j]$ is
a reparametrization of the orbit segment $\{\Phi^tq_j\mid
t\in [0,t_j]\}$ by a translation.
The points $q_1,\dots,q_{k-1}$ are the
breakpoints of the characteristic arc.

Let $\tilde\gamma$ be a lift of $\gamma$ to ${\cal Q}^1(S)$.
By Lemma \ref{quasigeodesic}, the assignment
$t\to \Upsilon_{\cal T}(P\tilde \gamma(t))$ is 
an $\ell_0$-quasi-geodesic
in ${\cal C}(S)$ for a number
$\ell_0>1$ only depending on $K$. The map 
$t\to P\tilde \gamma(t)\in ({\cal T}(S),d_T)$ is one-Lipschitz.
Since by inequality (\ref{upsilonlipschitz})
there is a number $L>0$ such that
$d_T(Pq,Pz)\geq d(\Upsilon_{\cal T}(Pq),\Upsilon_{\cal T}(Pz))/L-L$
for all $q,z\in {\cal Q}^1(S)$, we conclude that the curve
$t\to P\tilde\gamma(t)$ is a uniform quasi-geodesic in
${\cal T}(S)$. By the first part of 
Theorem \ref{teichcurvecomp}, this implies
that there is a Teichm\"uller geodesic whose  
Hausdorff distance to $P\tilde\gamma$
is bounded from above by a universal
constant.  As a consequence, the Hausdorff
distance between $\tilde \gamma$ and the tangent line
of this geodesic is bounded from above by a universal constant 
$b>0$.

A mapping class $g\in {\rm Mod}(S)$ is
pseudo-Anosov if the cyclic subgroup of ${\rm Mod}(S)$
generated by $g$ acts on the curve graph ${\cal C}(S)$ with
unbounded orbits. In this case the conjugacy class of $g$ can be
represented by a closed orbit for the Teichm\"uller flow $\Phi^t$
on ${\cal Q}^1(S)/{\rm Mod}(S)$, and it
can be represented by a closed orbit for the
Teichm\"uller flow on $\hat {\cal Q}(S)$ if
the conjugacy class of $g$ is contained in the 
normal subgroup $\Gamma$ of ${\rm Mod}(S)$.
Assume now that the
$(n_0,\epsilon_0)$-pseudo-orbit 
$q_0,\dots,q_k$ contained in $K$ is closed. Let
$\tilde \gamma$ be a lift to ${\cal Q}^1(S)$
of a closed characteristic arc
$\gamma$ for the pseudo-orbit. 
By the first part of Lemma \ref{quasigeodesic}, the curve $t\to
\Upsilon_{\cal T}(P\tilde \gamma(t))$ is an infinite
$\ell_0$-quasi-geodesic in ${\cal C}(S)$ which is invariant under an
element $g\in \Gamma<{\rm Mod}(S)$ of the mapping class group. The mapping
class $g$ acts on this quasi-geodesic as a translation and hence
it is pseudo-Anosov. As a consequence, there is a unique
$g$-invariant Teichm\"uller geodesic in ${\cal T}(S)$ whose
cotangent line in ${\cal Q}^1(S)$ projects to a periodic orbit
of $\Phi^t$ in $\hat{\cal Q}(S)$ which 
defines the free
homotopy class of $\gamma$. In other words, there is a closed orbit for
$\Phi^t$ in $\hat{\cal Q}(S)$ 
which is freely homotopic to $\gamma$.

By the first part of Theorem \ref{teichcurvecomp},
applied to the biinfinite quasi-geodesic $P\tilde \gamma$
in ${\cal T}(S)$, this orbit
is contained in a compact subset 
$C_0\supset K$ of $\hat{\cal Q}(S)$ 
not depending on the pseudo-orbit. Moreover, it
$b$-shadows the pseudo-orbit for a number $b>0$ only
depending on $K,n_0,\epsilon_0$. This shows the
first part of Theorem \ref{Anosov}.

Let $C\subset C_0$ be the $\Phi^t$-invariant
subset of $C_0$ of all points
whose $\Phi^t$-orbit is entirely contained in $C_0$. 
The periodic orbit defined by the conjugacy class of 
the pseudo-Anosov element $g$
is contained in $C$.
Let $\tilde C$ be the preimage of $C$ in ${\cal Q}^1(S)$.
Then every lift to ${\cal Q}^1(S)$ of a periodic
orbit in $\hat{\cal Q}(S)$ determined as above by
a closed $(n_0,\epsilon_0)$-pseudo-orbit contained in $K$
is contained in $\tilde C$.

Let again $\pi:{\cal Q}^1(S)\to {\cal P\cal M\cal L}$
be the canonical projection.
Write $A=\pi(\tilde C\cup {\cal F}(\tilde C))$ where
${\cal F}:{\cal Q}^1(S)\to {\cal Q}^1(S)$ is the flip
$q\to {\cal F}(q)=-q$.
As in Section 3, let $F_A=F\vert A:A\to \partial{\cal C}(S)$ be the
measure forgetting injection. 
For $q\in \tilde C\cup {\cal F}(\tilde C)$ let 
$\delta_{Pq}$ be the distance on $\partial{\cal C}(S)$ defined
in equation (\ref{distance}) and 
denote by $D_q(\pi(q),r)$ the
ball of radius $r$ about $\pi(q)$ in $A$ with
respect to the distance 
$(x,y)\in A\times A\to \delta_{Pq}(F_Ax,F_Ay)\in [0,\infty)$
which we denote again by $\delta_{Pq}$ (compare the 
proof of Theorem \ref{expansive}).

By the second part of 
Theorem \ref{teichcurvecomp},  applied to the projection into
${\cal T}(S)$ of the preimage $\tilde C$ of the compact
$\Phi^t$-invariant set $C\subset \hat{\cal Q}(S)$,
by Lemma \ref{expansion} and
by inequality (\ref{deltacompare}) of Section 2,
there are numbers $\alpha_0<1/2,
a>1,b>0$ such that for every $q\in \tilde C$ and
for all $t>0$ we have
\begin{equation}\label{cont1}
\delta_{P\Phi^{-t}q}\leq ae^{-bt}\delta_{Pq}\text{ on }
D_q(\pi(q),4\alpha_0).\end{equation}
Moreover, for every $\alpha<\alpha_0$
there is a number $\beta=\beta(\alpha)<1$ such that for
every $q\in \tilde C$ we have $A\cap \pi B(q,\beta)\subset
D_q(\pi(q),\alpha)$ where $B(q,\beta)$ is the ball of
radius $\beta$ about $q$ in ${\cal Q}^1(S)$
(compare the proof of Theorem \ref{expansive}).

Let $n=\max\{n_0, \log(4a)/b\}$,
let $\alpha<\alpha_0$ and let $\sigma=\min\{\epsilon_0,\beta(\alpha),
\kappa^{-1}\log 2\}$ where $\kappa>0$ is as in
inequality (\ref{deltacompare}).
We claim that for a lift $\tilde \gamma:[0,T]\to {\cal Q}^1(S)$ of a
characteristic arc $\gamma$ of any $(n,\sigma)$-pseudo-orbit 
contained in $K$ we have
\[\delta_{P\tilde \gamma(0)}(\pi\tilde \gamma(0),\pi\tilde
\gamma(T))\leq \alpha.\] 

To see this we proceed by induction on the
number of breakpoints of the pseudo-orbit. The case that there is
a no breakpoint is trivial, so assume that the claim is known
whenever the number of breakpoints of the pseudo-orbit is at most
$k-1\geq 0$. Let $\tilde\gamma$ be a lift to ${\cal Q}^1(S)$ 
of a characteristic arc $\gamma$ 
of an $(n,\sigma)$-pseudo-orbit contained in $K$ 
with $k$ breakpoints. Let $t_0\geq n+1$ be
such that $\gamma(t_0)$ is the first breakpoint of $\gamma$.
By assumption and the
choice of the parametrization of a characteristic arc we have
$d(\tilde \gamma(t_0),\tilde \gamma(t_0-1))\leq \sigma$. Since
$\tilde \gamma(t_0)\subset\tilde K\subset 
\tilde C,\tilde \gamma(t_0-1)\in \tilde C$,
by the choice of $\sigma$ we have
\begin{equation}
\delta_{P\tilde \gamma(t_0)}(\pi\tilde \gamma(t_0),\pi\tilde \gamma(t_0-1))
\leq \alpha,
\end{equation}
moreover the distances
$\delta_{P\tilde \gamma(t_0)}, \delta_{P\tilde \gamma(t_0-1)}$ are
$2$-bilipschitz equivalent (recall that the projection
$P:{\cal Q}^1(S)\to {\cal T}(S)$ is distance non-increasing).

Now $\pi\tilde \gamma(T)\in D_{\tilde
\gamma(t_0)}(\pi\tilde \gamma(t_0),\alpha)$
by the induction hypothesis and therefore
\begin{equation}
\delta_{P\tilde \gamma(t_0-1)}(\pi\tilde \gamma(t_0-1),
\pi\tilde\gamma(T))\leq 4\alpha.
\end{equation}
On the other hand, since $\alpha\leq \alpha_0$, 
since $n\geq \log(4a)/b$ and since
$\pi\tilde \gamma(t_0-1)=\pi\tilde \gamma(0)$
we infer from the estimate (\ref{cont1}) that
\begin{equation}\delta_{P\tilde
\gamma(t_0-1)}(\pi\tilde \gamma(t_0-1), \pi\tilde \gamma(T))\geq
4 \delta_{P\tilde \gamma(0)} (\pi\tilde \gamma(0),\tilde
\gamma(T)).\end{equation}
Together this implies the claim.

Note that 
the argument in the previous paragraph together with the estimate
(\ref{cont1}) also shows that 
\begin{equation}\label{internalestimate}
\delta_{\tilde \gamma(t)}
(\pi\tilde \gamma(t),\pi\tilde \gamma(T))\leq 4a\alpha
\end{equation}
for \emph{all} $t\in [0,T]$ with the additional 
property that $\tilde \gamma(t)\in \tilde C$ (namely
this holds true for every $t$ such that $\tilde\gamma(t)$
projects to an orbit segment defining the pseudo-orbit).

Let again $\tilde \gamma$ be a biinfinite lift 
to $\tilde {\cal Q}^1(S)$
of a closed characteristic curve $\gamma$ for an $(n,\sigma)$-pseudo-orbit
contained in $K$. The curve
$\Upsilon_{\cal T}(P\tilde \gamma)$ is a uniform
quasi-geodesic in ${\cal C}(S)$ which is invariant under 
a pseudo-Anosov element
$g\in \Gamma<{\rm Mod}(S)$ on 
$\partial{\cal C}(S)$. The conjugacy class of $\gamma$ 
defines the free homotopy class
of $\gamma$. The oriented 
cotangent line of the axis of $g$ is contained
in $\tilde C$. If $z\in \tilde C$ is a point in this
cotangent line then $\pi(z)\in A$ is a fixed point for the
action of $g$ on ${\cal P\cal M\cal L}$. 
An inductive application 
to longer and longer subsegments of $\tilde \gamma$
of the argument which lead to 
the estimate (\ref{internalestimate})
shows
that for every $t\in \mathbb{R}$ such that 
$\tilde \gamma(t)\subset \tilde K\subset
\tilde C$ 
the fixed point $\pi(z)\in A$ of $g$ is contained in the
ball $D_{\tilde \gamma(t)}(\pi(\tilde \gamma(t)),4a\alpha)$.
The same argument also shows that
the fixed point $\pi(-z)$ for the action of $g$ 
is contained
in $D_{-\tilde \gamma(t)}(\pi(-\tilde \gamma(t)),4a\alpha)$.
The periodic orbit on $\hat{\cal Q}(S)$ defined
by $g$ is contained in the compact $\Phi^t$-invariant 
subset $C\supset K$ of $\hat{\cal Q}(S)$
determined above.

By the considerations in 
Theorem \ref{expansive} and its proof, applied
to the compact $\Phi^t$-invariant subset $C$ of
$\hat{\cal Q}(S)$, 
this means that
for every $\delta >0$ there is a constant
$\beta>0$ only depending on $K$ with the following property.
Let $q_0,\dots,q_k$ be a closed $(n,\beta)$-pseudo-orbit
contained in $K$.
Then there is a closed orbit for $\Phi^t$ contained in $C$ whose 
Hausdorff distance
to a closed characteristic curve defined by the
pseudo-orbit is at most $\delta$. 
From this Theorem \ref{Anosov} follows.
\qed

The Anosov closing lemma implies the
existence of many periodic orbits near any
non-wandering point of a compact $\Phi^t$-invariant
subset $K$ of $\hat{\cal Q}(S)$. However, as for
compact invariant hyperbolic sets in the usual sense of 
smooth dynamical systems (see \cite{HK95}), 
these periodic orbits are in general not contained in $K$.
The next corollary is an immediate adaptation of Corollary 6.4.19 of
\cite{HK95} and shows that the periodic orbits
can be chosen to be contained in an arbitrarily small
neighborhood of $K$.

\begin{corollary}\label{accumulation}
Let $K$ be a compact $\Phi^t$-invariant
subset of $\hat{\cal Q}(S)$ and let $U$ be an open
neighborhood of $K$.
Then every non-wandering point $q\in K$ is an
accumulation point of periodic points
of $\Phi^t$ whose orbits are entirely contained in $U$.
\end{corollary}

In the case of a topologically transitive
compact invariant set $K\subset \hat{\cal Q}(S)$ we can say more.

\begin{lemma}\label{dense}
Let $K$ be a compact $\Phi^t$-invariant topologically transitive
subset of $\hat{\cal Q}(S)$. Then for every 
$\sigma >0$ there is a periodic orbit for $\Phi^t$ whose
Hausdorff-distance to $K$ (as subsets of $\hat{\cal Q}(S)$)
is at most $\sigma$.
\end{lemma}
\begin{proof}
Let $K\subset \hat{\cal Q}(S)$ be a compact 
$\Phi^t$-invariant topologically
transitive set and let $\sigma >0$. Let $n=n(K)>0$,
$\epsilon_2=\epsilon_2(K,\sigma/2)<\sigma/2$ 
be as in Theorem \ref{Anosov}.
Since $K$ is topologically transitive 
by assumption, there is some $q\in K$ and 
there is some $T>n$ such that
$d(q,\Phi^Tq)<\epsilon_2$ and that moreover
the Hausdorff distance between the set $K$ and its subset
$B=\{\Phi^tq\mid 0\leq t\leq T\}$ is at most $\sigma/2$.

By Theorem \ref{Anosov}, applied to 
the closed $(n,\epsilon_2)$-pseudo-orbit
defined by the orbit segment $\{\Phi^tq\mid 0\leq t\leq T\}$,
there is a periodic orbit
for $\Phi^t$ whose Hausdorff distance to 
$B$ is at most $\sigma/2$. This means that the
Hausdorff distance between this orbit and
the set $K$ is at most $\sigma$ and shows the lemma.
\end{proof}

For a compact $\Phi^t$-invariant subset
$K\subset \hat{\cal Q}(S)$ denote 
by $h_{\rm top}(K)$ the topological
entropy of the restriction of $\Phi^t$ to $K$. For an
arbitrary subset $U\subset \hat{\cal Q}(S)$ and a number
$r>0$ let $n_U(r)$ be the number of all periodic orbits
of $\Phi^t$ of period at most $r$ which are contained in $U$.
The following corollary is another fairly immediate consequence
of Theorem \ref{Anosov}. Together with Corollary \ref{lowerbound} it 
shows Theorem \ref{thm3} from the introduction.

\begin{corollary}\label{topup}
Let $K\subset \hat{\cal Q}(S)$ be a
compact $\Phi^t$-invariant topologically transitive set.
Then for every open neighborhood $U$ of $K$ we have
\[h_{\rm top}(K)\leq \lim\inf_{r\to \infty}\frac{1}{r}\log n_U(r).\]
\end{corollary}

\begin{proof}
Let $K\subset \hat{\cal Q}(S)$ be a
topologically transitive compact $\Phi^t$-invariant
set and let $U$ be an open neighborhood of $K$.
Then there is a number $\beta>0$ such that $U$ contains
the $\beta$-neighborhood of $K$. 

Let $\delta<\beta$ be sufficiently small that
the $\delta$-neighborhood of every point in $K$ 
is contained in a contractible subset of $\hat{\cal Q}(S)$.
Let $n=n(K)>0,\epsilon_2=\epsilon_2(K,\delta/8)<1$ be as in 
Theorem \ref{Anosov}.
Since the Teichm\"uller flow on $K$ is topologically
transitive by assumption, by compactness of $K\times K$ 
there is a number $N>n$ with the
following property. Let $q,q^\prime\in K$; then there is some
$u \in K$ and some $T\in [n,N]$ 
with $d(u,q^\prime)<\epsilon_2$ and $d(\Phi^Tu,q)<\epsilon_2$.

A subset $E$ of $K$ is called
\emph{$(m,\delta)$-separated} for some $m\geq 0$ if
for any two points $q\not=u\in E$ we have
\begin{equation}
d(\Phi^tq,\Phi^tu)\geq \delta \text{ for some }t\in [0,m].
\end{equation}
Let $m>n$ and let $E_m\subset K$ be any 
$(m,\delta)$-separated set. Let $q\in E_m$. By the choice
of $N>n$ there is 
some $u\in K$ and some $T\in [n,N]$  such that 
$d(u,\Phi^mq)<\epsilon_2$ and
$d(\Phi^T u,q)<\epsilon_2$.
By Theorem \ref{Anosov}, the closed
$(n,\epsilon_2)$-pseudo-orbit
$q,u,q$ is $\delta/8$-shadowed by a periodic orbit which 
defines the characteristic free homotopy class of the pseudo-orbit.
Since periodic orbits for $\Phi^t$ in $\hat{\cal Q}(S)$ 
minimize the length in their free homotopy class,
the length of the periodic orbit does not exceed $m+N+2\epsilon_2$.
Moreover, by the choice of $\delta$ this periodic orbit 
is contained in $U$. There is a point $\zeta(q)$ on the 
orbit with $d(q,\zeta(q))\leq \delta/8$. In other words,
there is a map $\zeta$ which associates to every
point $q\in E_m$ a point $\zeta(q)\in U$ 
whose orbit under $\Phi^t$ is entirely contained in $U$ and 
is periodic of period at most
$m+N+2\epsilon_2$.

Since the points in the set $E_m$ are $(m,\delta)$-separated by
assumption and the orbits
of $\Phi^t$ are geodesics parametrized by arc length,
the orbit segments 
$c(q)=\cup_{t\in (-\delta/8,\delta/8)}\Phi^t\zeta(q)$  
$(q\in E_m)$ are pairwise disjoint. Thus for a fixed periodic
orbit $\gamma$ for $\Phi^t$ of length
at most $m+N+2\epsilon_2$ there are at most 
$4(m+N+2)/\delta$ distinct points $q\in E_m$ with
$\zeta(q)\in \gamma$. As a consequence, 
there are at least $\delta\,{\rm card}(E_m)/4(m+N+2)$ 
distinct periodic orbits
of period at most $m+N+2$ in $U$. This shows that  
the asymptotic growth as $m\to\infty$ 
of the maximal cardinality of 
an $(m,\delta)$-separated subset of $K$  
does not exceed the asymptotic
growth of the numbers $n_U(r)$ as $r\to \infty$. 
The corollary is now an immediate consequence from
the definition of the topological entropy of a continuous
flow on a compact space (recall also from
Theorem \ref{expansive} that the Teichm\"uller flow on
$K$ is expansive and hence for all sufficiently 
small $\delta>0$ its topological entropy is
just the asymptotic growth rate of maximal $(m,\delta)$-separated
sets as $m\to \infty$).
\end{proof}

\section{Lower bounds for the number of periodic orbits}

In this section we complete the proof of Theorem \ref{thm2}
from the introduction.
For this we continue to use the assumptions and
notations from Sections 2 and 3. In particular,
we always denote by
$d_T$ the Teichm\"uller metric on Teichm\"uller space
${\cal T}(S)$ for $S$. 

We begin with establishing the first part of Theorem \ref{thm2} which
is immediate from the work of Eskin and
Mirzakhani \cite{EM08}. Since the proof is 
short and easy, we include it for completeness.

The \emph{Poincar\'e series with exponent
$\alpha>0$} at a point
$x\in {\cal T}(S)$ 
is defined to be the series
\begin{equation}
\sum_{g\in {\rm Mod}(S)}e^{-\alpha d(x,gx)}.\end{equation}
The \emph{critical exponent} of ${\rm Mod}(S)$
is the infimum of all numbers $\alpha >0$
such that the Poincar\'e series with exponent
$\alpha$ converges. Note that this critical
exponent does not depend on the choice of $x$.
Athreya, Bufetov, Eskin and Mirzakhani \cite{ABEM06}
showed that the critical exponent of the Poincar\'e
series equals $h=6g-6+2m$ and that
the Poincar\'e series diverges at the critical exponent.

For $r>0$ and for
a compact set $K\subset {\cal Q}^1(S)/{\rm Mod}(S)$ let
$n^\cap_K(r)$ be the number of
all periodic orbits for the
Teichm\"uller flow of period at most $r$ which
\emph{intersect} $K$.
The next lemma is the first
part of Theorem 1.

\begin{lemma}\label{upperbound}
For every compact subset $K$ of
${\cal Q}^1(S)/{\rm Mod}(S)$ we have
\[\lim\sup_{r\to \infty}\frac{1}{r}\log n^\cap_K(r)\leq 6g-6+2m.\]
\end{lemma}
\begin{proof}
Let $\hat K$ be any compact subset
of the \emph{moduli space}
${\cal M}(S)={\cal T}(S)/{\rm Mod}(S)$ 
which is the closure of an open set. Let
$K_1\subset {\cal T}(S)$ be a relative compact
fundamental
domain for the action of ${\rm Mod}(S)$ on the
preimage $\tilde K$ of $\hat K$ in ${\cal T}(S)$. 
Let $D$ be the diameter
of $K_1$ and let $x\in K_1$ be any point.
Let $g\in {\rm Mod}(S)$ be a pseudo-Anosov element
whose \emph{axis} 
(i.e. the unique $g$-invariant Teichm\"uller geodesic
on which $g$ acts as a translation) 
projects to a closed geodesic $\gamma$ in
moduli space which intersects $\hat K$.
Then there is a point $\tilde x\in K_1$ which
lies on the axis of a conjugate of $g$ which
we denote again by $g$ for simplicity.
By the properties of an axis, the
length $\ell(\gamma)$ of the 
closed geodesic $\gamma$ equals
$d_T(\tilde x,g\tilde x)$. On the other hand,
we have
\begin{equation}\label{distlength}
d_T(x,gx)\leq d_T(\tilde x,g\tilde x)+2d_T(x,\tilde x)
\leq \ell(\gamma)+2D\end{equation}
by the definition of $D$, the
choice of $\tilde x$ and invariance of the Teichm\"uller metric under
the action of ${\rm Mod}(S)$.
Therefore, if we denote by $K\subset{\cal Q}^1(S)/{\rm Mod}(S)$ the
preimage of $\hat K\subset {\cal M}(S)$ under
the natural projection and if we define $N(r)$ for $r>0$
to be the number
of all $g\in {\rm Mod}(S)$ with
$d(x,gx)\leq r$, then
we have
\begin{equation}\lim\sup_{r\to \infty}
\frac{1}{r}\log n^\cap_K(r) \leq
\lim\sup_{r\to \infty}\frac{1}{r}\log N(r).
\end{equation}

Since the critical exponent of the
Poincar\'e series equals $6g-6+2m$, 
for every $\epsilon >0$ the Poincar\'e series converges at the
exponent $\alpha=6g-6+2m+\epsilon$. Let $c(\alpha)>0$ be its value.
Then for every $r >0$, the cardinality of the set
$\{g\in {\rm Mod}(S)\mid d(x,gx)\leq r\}$ does not
exceed $c(\alpha) e^{\alpha r}$ (note that the term
in the Poincar\'e series corresponding
to such an element of ${\rm Mod}(S)$ 
is \emph{not smaller} than $e^{-\alpha r}$).
This shows that $\lim\sup \frac{1}{r}\log N(r)\leq 6g-6+2m+\epsilon$.
Since $\epsilon >0$ and 
the compact set $\hat K\subset
{\cal M}(S)$ were arbitrarily chosen, the lemma follows.
\end{proof}

As an immediate consequence we obtain

\begin{corollary}\label{htopup2}
Let $K\subset {\cal Q}^1(S)/{\rm Mod}(S)$ be a compact
$\Phi^t$-invariant topologically transitive set.
Then $h_{\rm top}(K)\leq 6g-6+2m$.
\end{corollary}
\begin{proof}
Let $K\subset {\cal Q}^1(S)/{\rm Mod(S)}$ be a compact $\Phi^t$-invariant 
topologically
transitive set, let $q\in K$ be a point whose
orbit under $\Phi^t$ is dense in $K$ and let
$\hat q\in \hat{\cal Q}(S)$ be a preimage of 
$q$ under the natural projection $\Theta:\hat{\cal Q}(S)\to 
{\cal Q}^1(S)/{\rm Mod}(S)$. Let $\hat K$ be the closure of the
orbit of $\hat q$; then $\hat K$ is a compact
$\Phi^t$-invariant topologically transitive
set. By equivariance of the Teichm\"uller flow under the
projection $\Theta$, this set is 
mapped by $\Theta$
onto $K$. Moreover, by
Corollary \ref{topup}, for every open
relative compact neighborhood $U$ of $\hat K$ we have
\begin{equation}\label{uptopes}
h_{\rm top}(K)\leq h_{\rm top}(\hat K)\leq 
\lim\inf_{r\to \infty}\frac{1}{r}\log n_U(r).
\end{equation} 

Now the projection $\Theta$ 
maps periodic orbits for 
$\Phi^t$ in $U$ of period at most $r$ 
to periodic orbits for $\Phi^t$ of period at most $r$ 
which are contained in the relative compact set  
$\Theta(U)\subset {\cal Q}^1(S)/{\rm Mod}(S)$. If the periodic
orbits $\gamma_1\not=\gamma_2$ in $U$ are mapped
to the same periodic orbit in $\Theta(U)$ then there
is some element $g$ from the factor group 
$G={\rm Mod}(S)/\Gamma$ which maps $\gamma_1$ to $\gamma_2$.  
Since $G$ is finite, the number of distinct periodic orbits in
$U$ which are mapped to the single orbit 
in $\Theta(U)$ is uniformly bounded. Therefore by
Lemma \ref{upperbound} we have
\begin{equation}
\lim\inf_{r\to \infty}\frac{1}{r}\log n_U(r)
\leq \lim\inf_{r\to \infty}\frac{1}{r} \log
n_{\Theta(U)}(r)\leq 6g-6+2m.
\end{equation}
This shows the corollary.
\end{proof}

Now we are ready for the proof of 
the second part of Theorem 2 from the
introduction.

\begin{proposition}\label{final}
For every $\epsilon >0$ there is a compact 
$\Phi^t$-invariant subset $K$ of ${\cal Q}^1(S)/{\rm Mod}(S)$ with
\[\lim\inf_{r\to \infty}\frac{1}{r}\log n_K(r)\geq 6g-6+2m-\epsilon.\]
\end{proposition}
\begin{proof} 
As in Section 2, let ${\cal F\cal M\cal L}\subset {\cal P\cal M\cal L}$ 
be the ${\rm Mod}(S)$-invariant Borel subset of all 
projective measured geodesic
laminations whose support is minimal and fills up $S$ and let
$F:{\cal F\cal M\cal L}\to \partial{\cal C}(S)$
be the continuous
${\rm Mod}(S)$-equivariant  surjection
which associates to a projective measured geodesic
lamination in ${\cal F\cal M\cal L}$ its support.
Let $\pi:{\cal Q}^1(S)\to {\cal P\cal M\cal L}$ be 
the natural projection as defined in (\ref{pi}) and define
\begin{equation}
{\cal A}=\pi^{-1}{\cal F\cal M\cal L}\subset {\cal Q}^1(S).
\end{equation}

Let $\lambda$ be the $\Phi^t$-invariant probability measure
on ${\cal Q}^1(S)/{\rm Mod}(S)$ in the Lebesgue measure class constructed
in \cite{M82,V86}. This measure is ergodic and mixing
under the Teichm\"uller flow, with full support. In particular,
the $\Phi^t$-orbit of $\lambda$-almost every point 
$q\in {\cal Q}^1(S)/{\rm Mod}(S)$ 
returns to every neighborhood of $q$ for arbitrarily
large times. 
The measure $\lambda$ lifts to a
${\rm Mod}(S)$-invariant $\Phi^t$-invariant Radon
measure $\tilde \lambda$ on ${\cal Q}^1(S)$ of full support which
gives full measure to the ${\rm Mod}(S)$-invariant Borel 
set ${\cal A}$ \cite{M82}.

The Lebesgue measure $\tilde\lambda$ on ${\cal Q}^1(S)$
is absolutely continuous
with respect to the strong unstable foliation. More
precisely, for every $q\in {\cal Q}^1(S)$ 
there is a natural 
conditional measure $\tilde\lambda_q$ for $\tilde \lambda$
on the strong unstable
manifold $W^{su}(q)$, and these
conditional measures transform under the Teichm\"uller flow
via $d\tilde\lambda_{\Phi^tq}\circ \Phi^t=e^{ht}d\tilde\lambda_q$ where
$h=6g-6+2m$ as before. The image under the projection $\pi$ of the
measure $\tilde \lambda_q$ on $W^{su}(q)$ 
is a locally finite Borel measure $\lambda_q$ 
on the open dense subset of ${\cal P\cal M\cal L}$
of all projective measured geodesic laminations which together
with $\pi(-q)$ jointly fill up $S$. 
The measures $\lambda_q$ are all absolutely continuous, and they
depend continuously on $q\in {\cal Q}^1(S)$ in the weak$^*$-topology.
Moreover, for each $q$ the measure $\lambda_q$ gives full measure
to the set ${\cal F\cal M\cal L}$ and hence it can be
mapped via the surjection $F$ to a measure on $\partial{\cal C}(S)$
which we denote again by $\lambda_q$.

Recall from (\ref{distance}) and (\ref{deltacompare})
the definition of the distances 
$\delta_x$ $(x\in {\cal T}(S))$ on $\partial{\cal C}(S)$
and their properties.
For $q\in {\cal A}$ and $\chi>0$ define
\[D(q,\chi)\subset \partial{\cal C}(S)\] to be
the closed $\delta_{Pq}$-ball of radius
$\chi$ about $F\pi(q)\in \partial{\cal C}(S)$.
Note that $D(q,\chi)$ contains the image under
the map $F$ of the ball $B_{\delta_{Pq}}(\pi(q),\chi)$
used in Section 4. Note moreover that $D(q,\chi)$ depends on 
$q$ and not only on its center $F\pi(q)$ and the radius $\chi$
since for $t\not =0$ the distances $\delta_{Pq}$ and 
$\delta_{P\Phi^tq}$ do not coinicide,

Let $q_0\in {\cal Q}^1(S)/{\rm Mod}(S)$ be a typical point for 
the Lebesgue measure $\lambda$ (so that the Birkhoff ergodic
theorem holds true for $q_0$) 
and let $q_1$ be a lift of $q_0$ to ${\cal Q}^1(S)$. 
Assume without loss of generality that $Pq_1$ is not
fixed by any element of ${\rm Mod}(S)$. This is
possible since the set of points in ${\cal T}(S)$ which
are stabilized by a non-trivial element of ${\rm Mod}(S)$
is closed and nowhere dense and since the Lebesgue measure
is of full support.

Let $m>1$ be 
as in the second part of Lemma \ref{quasigeodesic}.
We may assume that the image under the map $\Upsilon_{\cal T}$ of
every Teichm\"uller geodesic in ${\cal T}(S)$ is an
unparametrized $m$-quasi-geodesic in ${\cal C}(S)$.
Since $q_0$ is a typical point for the 
Lebesgue measure, the unparametrized
quasi-geodesic $t\to \Upsilon_{\cal T}(P\Phi^tq_1)$ is 
of infinite diameter.

Let $\kappa>0$ be as in inequality (\ref{deltacompare}).
By Lemma \ref{expansion} and the inequality
(\ref{deltacompare}), there is a number
$\alpha>0$ depending on $m$ and 
there is a neighborhood $V$ of
$q_1$ in ${\cal Q}^1(S)$ of diameter at most
$\log 2/\kappa$ and a number $T_0 >0$
such that 
\begin{equation}\label{tdetermin}
\delta_{P\Phi^tu}\geq 16\delta_{Pu}\text{ on }
D(\Phi^tu,\alpha)\text{ for } t\geq T_0,\,u\in V\cap {\cal A}.
\end{equation}

Since $q_0$ is 
recurrent and hence the vertical measured geodesic lamination
of $q_1$ is uniquely ergodic and fills up $S$, by 
the second part of Lemma 3.2 of 
\cite{H09} there is a number $\chi\leq \alpha/4$ such that
\begin{equation}\label{vest}
F\pi(V\cap {\cal A}\cap W^{su}(q_1))\supset D(q_1,\chi).
\end{equation}

Let $\epsilon_0=\epsilon_0(q_0)>0$
be as in the second part of Lemma \ref{quasigeodesic}.
We may assume that the $\epsilon_0$-neighborhood of 
$q_0$ is contained in a contractible subset of 
${\cal Q}^1(S)/{\rm Mod}(S)$.
By continuity, there
is a compact
neighborhood $K\subset V$ of $q_1$ with the following
properties.
\begin{enumerate}
\item The diameter of $K$ does not exceed 
$\max\{\epsilon_0, (\log 2)/\kappa\}$.
\item $F\circ \pi(K\cap {\cal A})\subset D(q_1,\chi/4)$.
\end{enumerate}
By the second requirement for $K$, if
$q,u\in K\cap {\cal A}$ 
then $\delta_{Pq_1}(F\pi(u),F\pi(q))\leq \chi/2$. The first
property of $K$ together with the
relation (\ref{deltacompare}) for the
distances $\delta_x$ $(x\in {\cal T}(S))$ then implies that
$\delta_{Pu}(F\pi(q),F\pi(u))\leq \chi$ and 
$D(q,\chi)\subset D(u,4\chi)$.

Following \cite{F69}, a \emph{Borel covering
relation} for a Borel subset $C$ of a topological 
space $X$ is a family ${\cal V}$ of
pairs $(x,V)$ where $V\subset X$ is a Borel set,
where $x\in V$ and such that
\begin{equation}
C\subset \bigcup \{V\mid (z,V)\in {\cal V}\text{ for some }
z\in C\}.\end{equation}

For $\chi>0$ and the neighborhood 
$K\subset {\cal Q}^1(S)$ of $q_1$ as above define
\begin{align}\label{vitalirelation}
{\cal V}_{q_0,\chi,K}=& \{(F\pi(q),g D(q_1,\chi))\mid \\
& q\in W^{su}(q_1)\cap {\cal A},g\in 
{\rm Mod}(S),gK\cap \cup_{t>0}\Phi^tq\not=
\emptyset\}.\notag
\end{align}
By Proposition 3.5 of \cite{H09}, via possibly decreasing the
size of $\chi$ and $K$ we may assume that 
the covering relation ${\cal V}_{q_0,\chi,K}$ is a 
\emph{Vitali relation} for the measure $\lambda_{q_1}$ on
$\partial {\cal C}(S)$.
In our context, this means that for every $T>0$ 
there is a covering of $\lambda_{q_1}$-almost all of 
$D(q_1,\chi/4)$ by pairwise disjoint sets from the relation 
of the form 
\[V(g,t)=(F\pi(u),gD(q_1,\chi))\] where
$u\in W^{su}(q_1)\cap {\cal A},F\pi(u)\in D(q_1,\chi/4)$, 
$g\in {\rm Mod}(S)$ and where $t\geq T$ is such that
$\Phi^tu\in gK$ (we refer to Section 3 of \cite{H09} for a detailed
discussion).

Since the measures $\lambda_q$ and
the distances $\delta_{Pq}$ on $\partial {\cal C}(S)$
depend continuously on 
$u\in {\cal Q}^1(S)$, there is a number 
$a\leq \lambda_{q_1}D(q_1,\chi/4)$ such that
$\lambda_qD(u,\chi)\in [a,a^{-1}]$ for all 
$q\in K,u\in K\cap {\cal A}$. By the transition properties
for the measures $\lambda_u$ and
invariance under the action of the mapping class group,
if $g\in {\rm Mod}(S)$, if $u\in W^{su}(q_1)$ and if
$t>0$ are such that $\Phi^tu\in gK$ for some
$t>0$ then $\lambda_{q_1}(gD(q_1,\chi))\in 
[ae^{-ht},e^{-ht}/a]$ (compare the discussion in \cite{H09}).
   
Let $n_0=n_0(q_0)>0$ be as in the second part of 
Lemma \ref{quasigeodesic}.
Let $\epsilon >0$ and let 
$T(\epsilon)>\max\{T_0,n_0\}+2$ be sufficiently large
that 
\[\int_{T(\epsilon)}^\infty e^{-\epsilon s}ds\leq e^{-2h}a^2.\]
Choose a 
covering of $\lambda_{q_1}$-almost all of 
$D(q_1,\chi/4)$ by pairwise disjoint sets 
$V(g,t)=(F\pi(u),gD(q_1,\chi))$
from the relation 
${\cal V}_{q_0,\chi,K}$ where
$u\in W^{su}(q_1)\cap {\cal A},F\pi(u)\in D(q_1,\chi/4)$, 
$g\in {\rm Mod}(S)$ and where $t\geq T(\epsilon)$ is such that
$\Phi^tu\in gK$. 
By the inclusion (\ref{vest}), we have $u\in V$ and therefore
from the assumption $T(\epsilon)\geq T_0$ 
and the estimate (\ref{tdetermin}) we deduce that
$D(\Phi^tu,4\chi)\subset D(u,\chi/4)\subset D(q_1,\chi)$.
On the other hand, we have $\Phi^tu\in gK$ 
and hence 
\begin{equation}\label{nesting}
gD(q_1,\chi)\subset
D(\Phi^tu,4\chi)\subset D(q_1,\chi).
\end{equation}

The total $\lambda_{q_1}$-mass of the balls from the covering is
at least $\lambda_{q_1}D(q_1,\chi/4)\geq a$. Therefore 
there is a number $T>T(\epsilon)+2$ such that
the total volume of those balls $V(g,t)$ 
from the covering which 
correspond to a parameter $t\in [T-2,T-1]$ 
is at least $e^{-\epsilon T}e^{2h}/a$. Now the 
$\lambda_{q_1}$-volume of each such ball is at most
$e^{-h(T-2)}/a$ and hence   
the number of these balls is at least 
$e^{(h-\epsilon)T}$.

Let $\{g_1,\dots,g_k\}\subset {\rm Mod}(S)$ 
be the subset of ${\rm Mod}(S)$ defining these balls.
By (\ref{nesting}) above, for any $i,j$ we
have
\[g_jg_iD(q_1,\chi)\subset g_jD(q_1,\chi).\]
Thus the sets $g_jg_iD(q_1,\chi)$ $(i,j=1,\dots,k)$ are
pairwise disjoint. Namely, the sets 
$g_jD(q_1,\chi)$ $(j=1,\dots,k)$ are pairwise disjoint,
and for each $j$ the sets 
$g_jg_iD(q_1,\chi)\subset g_jD(q_1,\chi)$ 
$(i=1,\dots,k)$ are pairwise disjoint
as well.
By induction, we conclude that for any two distinct words
$w_1=g_{i_1}\cdots g_{i_\ell}$ and $w_2=g_{j_1}\cdots g_{j_m}$
in the letters $g_1,\dots,g_k$, viewed as elements of
${\rm Mod}(S)$, the images of $D(q_1,\chi)$ under $w_1,w_2$ are
either disjoint or properly contained in each other.
This shows that the elements $g_1,\dots,g_k$ generate a
free semi-subgroup $\Lambda$ of ${\rm Mod}(S)$.

Since $T(\epsilon)\geq n_0$, each word $w$ 
of length $\ell \geq 1$ in the letters $g_1,\dots,g_k$ defines
a closed $(n_0,\epsilon_0)$-pseudo-orbit $u_0,\dots, u_\ell$ 
in ${\cal Q}^1(S)/{\rm Mod}(S)$ with
$d(u_i,q_0)<\epsilon_0$. This pseudo-orbit consists of 
the successive projections
to ${\cal Q}^1(S)/{\rm Mod}(S)$ of flow lines
$\{\Phi^tu\mid t\in [0,\tau]\}$ where $u\in V\cap W^{su}(q_1)\cap {\cal A}$
and $\tau\in [T-2,T-1]$ are
such that $F\pi(u)\in g_jD(q_1,\chi/4)$ for some $j\leq k$ and 
$\Phi^\tau u\in g_jK$. 
Thus by the second part of 
Lemma \ref{quasigeodesic}, if $\tilde \gamma$ is 
a lift to ${\cal Q}^1(S)$ of a characteristic
arc of such a pseudo-orbit then $\Upsilon_{\cal T}(\tilde \gamma)$
is a biinfinite unparametrized $m$-quasi-geodesic
in ${\cal C}(S)$ which is
invariant under the element of $\Lambda\subset {\rm Mod}(S)$ defined by $w$.
In particular, this element
is pseudo-Anosov, and its conjugacy class defines
the characteristic free homotopy
class of the closed pseudo-orbit. 

The length of the
periodic orbit of $\Phi^t$ determined by $w$ 
does not exceed the length of a characteristic closed
curve for the pseudo-orbit and hence its is not bigger than
$T\ell$.
Moreover, since by the choice of $n_0$ for any $s<t$ 
with the property that $\tilde \gamma(s),\tilde \gamma(t)$ 
project to distinct breakpoints of $\gamma$ the distance
between $\Upsilon_{\cal T}(\tilde \gamma(s)),
\Upsilon_{\cal T}(\tilde \gamma(t))$ is at least
$2c(m)$, it follows from Lemma 2.4 of \cite{H10} that 
the unparametrized $m$-quasi-geodesic $\Upsilon_{\cal T}(\tilde \gamma)$ 
is in fact
a \emph{parametrized $p$-quasi-geodesic} for some 
$p>m$. Using once more the first part of 
Theorem \ref{teichcurvecomp}, this implies
that the axis of the element of $\Lambda\subset {\rm Mod}(S)$ defined by
$w$ passes through a fixed compact
neighborhood $B$ of $Pq_1$ in ${\cal T}(S)$, and the projection 
of its unit tangent line to ${\cal Q}^1(S)/{\rm Mod}(S)$ 
is a periodic orbit
for $\Phi^t$ which is contained in 
a compact subset $C_0$ of ${\cal Q}^1(S)/{\rm Mod}(S)$ not depending on $w$.
If we denote by $C$ the closed subset of $C_0$ of all points
$z\in C_0$ whose orbit under $\Phi^t$ is entirely contained in 
$C_0$ then each of these orbits is contained in 
$C$. 

The above argument does not immediately imply that
the asymptotic 
growth rate of the number of periodic orbits in $C$ is
at least $h-\epsilon$. Namely,
periodic orbits of the Teichm\"uller flow on
${\cal Q}^1(S)/{\rm Mod}(S)$ correspond to
\emph{conjugacy classes} of pseudo-Anosov elements
in ${\rm Mod}(S)$. Thus if we want to count periodic
orbits for $\Phi^t$ in ${\cal Q}^1(S)/{\rm Mod}(S)$ using the 
semi-subgroup $\Lambda$ of ${\rm Mod}(S)$ constructed above, 
then we have to identify those elements of $\Lambda$ 
which are conjugate in ${\rm Mod}(S)$.

For this recall that the axis of each element
of the semi-subgroup $\Lambda$ of ${\rm Mod}(S)$ passes
through the fixed compact neighborhood $B$ of $Pq_1$.
Thus if $\gamma,\zeta$ is the axis of 
$v,w\in \Lambda$ and if $v,w$ are conjugate
in ${\rm Mod}(S)$ then 
there is some $b\in {\rm Mod}(S)$ with 
$w=b^{-1}vb$ and the following additional property.
Let $\gamma[0,\tau]$ be a fundamental domain for the
action of $v$ on $\gamma$ 
and such that $\gamma(0)\in B$. 
Such a fundamental domain always exists, perhaps after
a reparametrization of $\gamma$.
Then there is some $t\in [0,\tau]$ such that
$b^{-1}\gamma(t)\in B$. 

As a consequence, the number of all elements $w\in \Lambda$ which
are conjugate to a fixed element $v\in \Lambda$ is 
bounded from above by the number of elements
$b\in {\rm Mod}(S)$ with $bB\cap \gamma[0,\tau]\not=\emptyset$.
In particular, if $D$ is the diameter of $B$ then
this number does not exceed the cardinality of the
set 
\begin{equation}
\{b\in {\rm Mod}(S)\mid d_{\cal T}(bPq_1,\gamma[0,\tau])\leq D\}.
\end{equation}  
However, this cardinality is bounded from
above by a universal multiple of $\tau$.
Therefore there is a constant $c>0$ such that
for all sufficiently large $r>0$ 
the number of periodic
orbits of $\Phi^t$ contained in $C$ of length at most
$r$ is not smaller than $e^{(h-\epsilon)r}/cr$. 
This completes the proof of the proposition.
\end{proof}

{\bf Remarks:}

1. Proposition \ref{final} is equally valid, with identical
proof, for the Teichm\"uller flow $\Phi^t$ on $\hat{\cal Q}(S)$.
Together with Theorem \ref{thm3} 
it implies that the metric entropy $h$ of the
unique $\Phi^t$-invariant Lebesgue measure 
on $\hat{\cal Q}(S)$ in the Lebesgue measure class 
equals the supremum of the topological entropies of the
restriction of $\Phi^t$ to compact invariant subsets
of $\hat{\cal Q}(S)$. In \cite{BG07}, this fact was established
for the Teichm\"uller flow on the moduli space of abelian 
differentials using symbolic dynamics.
 
2. The abundance of orbits of the Teichm\"uller flow
which entirely remain in some compact set (depending on the orbit)
was earlier
established by Kleinbock and Weiss \cite{KW04}. They show that
this set is of full Hausdorff dimension.

\bigskip

{\bf Acknowledgement:} This work was
done in spring 2007
while I visited the Bernoulli center in Lausanne.
I thank Alain Valette for inviting me there, and
I thank the center for its hospitality
and the excellent working conditions it provided.

\bigskip

\noindent
MATHEMATISCHES INSTITUT DER UNIVERSIT\"AT BONN\\
ENDENICER ALLEE 60\\
D-53115 BONN, GERMANY\\
e-mail: ursula@math.uni-bonn.de

\end{document}